\documentclass[a4paper,reqno]{amsart}

\usepackage{amsmath}
\usepackage{paralist}
\usepackage{graphics}
\usepackage{epsfig}
\usepackage{amsfonts}
\usepackage{amssymb}
\usepackage{hyperref}
\usepackage{epstopdf}
\usepackage{color}
\usepackage{bm}
\usepackage[usenames,dvipsnames,svgnames]{xcolor}
\definecolor{refkey}{rgb}{1,0,0.5}
\definecolor{labelkey}{rgb}{0,0.4,1}
\usepackage{todonotes}
\usepackage{marginnote}

\newtheorem{theorem}{Theorem}[section]
\newtheorem{lemma}[theorem]{Lemma}

\def\x#1{(\ref{#1})}

\def\ifl{\iffalse }

\def\bc{\begin{center}}       \def\ec{\end{center}}
\def\ba{\begin{array}}        \def\ea{\end{array}}
\def\be{\begin{equation}}     \def\ee{\end{equation}}
\def\bea{\begin{eqnarray}}    \def\eea{\end{eqnarray}}
\def\beaa{\begin{eqnarray*}}  \def\eeaa{\end{eqnarray*}}

\numberwithin{equation}{section}

\newtheorem{definition}[theorem]{Definition}

\newtheorem{remark}[theorem]{Remark}
\numberwithin{equation}{section}

\begin{document}
	\author{Jos\'{e} Antonio Carrillo$^{1}$}
	\thanks{$^{1}$Mathematical Institute, University of Oxford, Oxford, OX2 6GG, UK. Email: carrillo@maths.ox.ac.uk}

	\author{Ke Lin$^{2}$}\thanks{$^{2}$School of Mathematics, Southwestern University of Finance and Economics, Chengdu, 611130, P.R. China. E-mail: linke@swufe.edu.cn}
	
	\title[Sharp critical mass for two-species chemotaxis system]
	{Sharp critical mass criteria for weak solutions to a degenerate cross-attraction system}

	

	
	\begin{abstract}
		The qualitative study of solutions to the coupled  parabolic-elliptic chemotaxis system   with nonlinear diffusion for two species will be considered in the whole Euclidean space $\mathbb{R}^d$ ($d\geq 3$). It was proven in \cite{CK2021-ANA} that there exist two critical curves that separate the global existence and blow-up of weak solutions to the above problem. We improve this result by providing sharp criteria for the dicothomy: global existence of weak solution versus blow-up below and at these curves. Besides, there exist sharp critical masses of initial data at the intersection of the two critical lines, which extend the well-known critical mass phenomenon in one-species Keller-Segel system in \cite{BCL09-CVPDE} to two-species case.
		
	\end{abstract}

	\maketitle
	
	\section{Introduction}
We analyse the Cauchy problem for a two-species chemotaxis system with nonlinear diffusion written as the parabolic-elliptic system given by
	\be \label{system}\begin{cases}
		\partial_tu_i(x,t)=  \Delta u^{m_i}_i(x,t)-\nabla \cdot (u_i(x,t)\nabla v_j(x,t)), &x\in \mathbb{R}^d, t>0, \\[0.2cm]
		-\Delta v_j(x,t)=u_j(x,t), &x\in \mathbb{R}^d, t>0, \\[0.2cm]
		u_i(x,0)=u^0_i(x), \,\,\,i,j=1,2,\,\,i\neq j,\,\,& x\in \mathbb{R}^d,
	\end{cases}
	\ee
where $m_i>1$, $d\geq 3$. This model arises in mathematical biology as a system for the communication between  breast tumor cells and macrophages via a short-ranged chemotactic signalling loop. In this setting, $u_1$ and $u_2$ respectively represent the density of macrophages and tumor cells, $v_1$ and $v_2$  respectively denote CSF-1 (colony stimulating factor 1) concentration and EGF (epidermal growth factor) concentration. Tumor cells secrete EGF, and respond to CSF-1, while macrophages secrete CSF-1, and respond to EGF. This signalling loop will cause tumor cells to aggregate potentially increasing the tumor growth, see more details in \cite{KPE14-JTB}. The choice of $m_i>1$ models volume size constraints for the migrations of cells similar to water filtration in porous medium. The solution $v_j$ to $(\ref{system})_2$ is given through the Newtonian potential of $u_j$:
	\begin{equation*} \label{newtonian potential for vi}
		\begin{split}
			v_j(x,t)=(\mathcal{K}\ast u_j)(x,t)=c_d\int_{\mathbb{R}^d}\frac{u_j(y,t)}{|x-y|^{d-2}}dy,\,\,j=1,2,
		\end{split}
	\end{equation*}
where $\mathcal{K}(x)=c_d/|x|^{d-2}$ and $c_d$ denotes the surface area of the sphere $\mathbb{S}^{d-1}$ in $\mathbb{R}^d$.
	
One of the most well-known biological systems to describe the aggregation of cellular slime molds, such as the Dictyostelium discoideum, was proposed by Keller and Segel \cite{KS1970-JTB} written in a simplified form as
	\be \label{KS system}\begin{cases}
		\partial_tu(x,t)=  \Delta u^{\kappa}(x,t)-\nabla \cdot (u(x,t)\nabla v(x,t)), &x\in \mathbb{R}^d, t>0, \\[0.2cm]
		-\Delta v(x,t)=u(x,t), &x\in \mathbb{R}^d, t>0,\\
		u(x,0)=u^0(x),\quad & x\in \mathbb{R}^d,
	\end{cases}
	\ee
where $u$ and $v$ denote the cell density and chemical concentration they secrete, respectively, $\kappa\geq 1$ is assumed to be the diffusion exponent, and $d\geq 2$. The initial mass $M$ of cells does not change during the evolution, i.e.,
	\begin{center}
		$M=\|u(t)\|_1=\|u^0\|_1$ for all $t>0$. 
	\end{center}
Note that the initial mass $M>0$ plays an important role in determining the behavior of solution to \eqref{KS system}. For the linear diffusion case in \eqref{KS system}, there exists a dichotomy between blow-up and global existence when $d=2$. If $M>8\pi$ the solutions blow up in finite time, while it was shown that there exists a global weak solution if $M<8\pi$ converging to a self-similar profile of \eqref{KS system}, see \cite{BDP2006-EJDE, CD2014-CPDE,CCLW2012-AAM, DP2004-CRMASP}. At the critical mass $M=8\pi$, the global-in-time existence of solutions together with corresponding asymptotic behavior as time $t\rightarrow\infty$ was given in \cite{BCC12-JFA,BCM2008-CPAM}.
	
In higher dimension $d\geq 3$, the similar  results for the situation $\kappa=m^*:=2-2/d$ have been shown  in \cite{BCL09-CVPDE,ST2009-DIE} to the case $\kappa=1$  and $d=2$. That is, there is a critical value on $M$ which distinguishes  global existence versus blow-up. One of our goal in this paper is to make sure that an analogous critical mass phenomenon occurs for two-population chemotaxis system \eqref{system} with $m_1=m_2=m^*$ in dimension $d\geq 3$. The number $m^*$ is chosen to produce a balance between the diffusion term $\Delta u^{\kappa}$ and the interaction term $\nabla \cdot (u\nabla v)$ in mass invariant scaling. Using this scaling argument, it is well-known that for $\kappa>m^*$, the diffusion term dominates and then it results in global existence of weak solutions for all masses \cite{S2006-DIE,Sugiyama2006-JDE}. Whereas if $\kappa<m^*$, the interaction term wins and finite-time blow-up solutions exist \cite{S2006-DIE}. The readers can find similar results in \cite{BRB11-Nonlinearity} for general singular kernels.
	
It should be pointed out that there exists another critical value on the $L^{m_*}$ norm of the initial data to classify the global existence of solution and the finite time blow-up of the solution if  $\kappa=m_*:=2d/(d+2)$ \cite{CLW2012-SJMA,Ogawa2011-DCDSS}. For the case of the intermediate exponent between $m_*$ and $m^*$ the criterion was examined in \cite{KNO2015-CVPDE,WCL2016-DCDS,CW2014-DM}. The characteristic of \eqref{KS system} for all ranges $\kappa>0$ has been extensively studied recently, including various aspects such as global existence, blow-up, hyper-contractive estimates and steady states, see \cite{BianLiu2013} for an excellent summary. The readers can find many references about similar arguments on the parabolic-parabolic case of \eqref{KS system}, see  \cite{IY2012-JDE2} for $1<\kappa<m^*$, \cite{WLC2019-ZAMP} for $m_*<\kappa<m^*$, \cite{BL2013-CPDE,LM2015-Poincare} for $\kappa=m^*$, and \cite{IY2012-JDE1,IY2020-DCDS} for $\kappa>m^*$.
	\begin{figure}[ht!]
		\label{figregimes}
		\centering
		\begin{tikzpicture}
			\draw (0,0) node[left] {$0$};
			\draw[->] (-0.2,0) --(8,0) node[right] {$m_1$};
			\draw[->] (0,-0.2) --(0,8) node[above] {$m_2$};
			\draw[dotted] (2,0)--(2,8);
			\draw[dotted] (0,2)--(8,2);
			\draw (2,0) node[below]{$1$};
			\draw (0,2) node[left]{$1$};	
			\draw[dotted] (3,0)--(3,3) node[below] at(3,0) {$m^*$};
			\draw[dotted] (0,3)--(3,3) node[left] at(0,3) {$m^*$};
			\draw[dotted] (4,0)--(4,2) node[below] at(4,0) {$d/2$};	
			\draw[dotted] (0,4)--(2,4) node[left] at(0,4) {$d/2$};			
			\draw[domain =2.3:3] plot (\x ,{1+2/(\x-2)});
			\draw[domain =3:4,red] plot (\x ,{1+2/(\x-2)});
			\draw[domain =2:4,blue] plot (\x ,{4/3+16/9/(\x-4/3) }) ;
			\node[right] at(2.5,6) {$m_1m_2+2m_1/d=m_1+m_2$};
			\draw[domain =3:8] plot (\x ,{2+2/(\x-1)}) node[right] at(6,2.6)  {$m_1m_2+2m_2/d=m_1+m_2$};
			\draw[domain =2:3,red] plot (\x ,{2+2/(\x-1)});
			\node[right] at(6,2.6)  {$m_1m_2+2m_2/d=m_1+m_2$};
			\draw (2.9,3.1) node[right]{${\bf{I}}:\,(m^*,m^*)$};
			\node[right] at(4,1.8) {${L_1}$};
			\node[above] at(1.8,4) {${L_2}$};
			\node[above] at(2.4,2.4) {${L_3}$};		
		\end{tikzpicture}
		\caption{Figure 1}
	\end{figure}
	
As for two-species model with cross-attraction chemical signals in \eqref{system}, the aggregation between two different species, interacting with each other in a chemotactic signalling loop, will affect the movements of individual species and the time asymptotic behavior of solutions, see  \cite{H2003,H2011-JNS} for  biological motivations and a survey of mathematical results on this aspects. The main tool used for the analysis of \eqref{system} depends on the existence of the free energy functional:
	\begin{equation} \label{definition of free energy F}
		\begin{split}
			\mathcal{F}[u_1,u_2](t)
			=&\underbrace{\frac{1}{m_1-1}\int_{\mathbb{R}^d}u^{m_1}_1(x,t)dx+\frac{1}{m_2-1}\int_{\mathbb{R}^d}u^{m_2}_2(x,t)dx}_{\text{diffusion terms}}\\
			&	-\underbrace{c_d\iint_{\mathbb{R}^d\times\mathbb{R}^d}\frac{u_1(x,t)u_2(y,t)}{|x-y|^{d-2}}dxdy}_{\text{cross-attraction term}}\\
			=&\sum^{2}_{i=1}\frac{1}{m_i-1}\|u_i\|^{m_i}_{m_i}-c_d\mathcal{H}[u_1,u_2],\quad t>0,
		\end{split}
	\end{equation}
which enjoys the following identity in the classical sense
	\begin{equation*} \label{ free energy identity}
		\begin{split}
			\mathcal{F}[u_1,u_2](t)+\int^t_0\mathcal{D}[u_1,u_2](s)ds=\mathcal{F}[u^0_1,u^0_2],\quad t>0,
		\end{split}
	\end{equation*}
with the dissipation $\mathcal{D}$ given by
	\begin{equation*} \label{definition of dissipation D}
		\begin{split}
			\mathcal{D}[u_1,u_2](t)=&\int_{\mathbb{R}^d}u_1(x,t)\left
			|\frac{m_1}{m_1-1}\nabla u^{m_1-1}_1(x,t)-\nabla v_2(x,t)\right|^2dx\\
			&+\int_{\mathbb{R}^d}u_2(x,t)\left
			|\frac{m_2}{m_2-1}\nabla u^{m_2-1}_2(x,t)-\nabla v_1(x,t)\right|^2dx,\quad t>0.
		\end{split}
	\end{equation*}
Note that for the linear diffusion case $m_1=m_2=1$, the diffusion terms in \eqref{definition of free energy F} would be replaced by $\sum^{2}_{i=1}\int_{\mathbb{R}^d}u_i\log u_idx$.
	
For linear cell diffusion of \eqref{system}, the critical mass phenomenon involved two masses of each species in two-dimensional setting has been detected in \cite{HT2019,HWYZ2019-Nonlinearity}. The main method used is the logarithmic Hardy-Littlewood-Sobolev inequality for system \cite{SW05-JEMS}.  We refer the readers to \cite{HT2019,KW2019-JDE,KW2020-EJAM,L2023,LZ2023,W2002-EJAM} for the literature on related multi-species model. When $m_i\in(1,d/2),i=1,2$, define
	\begin{equation*}\label{definition of p and q}
		p:=	\frac{d(m_1+m_2-m_1m_2)}{2m_2},\quad 	q:=\frac{d(m_1+m_2-m_1m_2)}{2m_1}.
	\end{equation*}
These numbers play a fundamental role in the analysis of \eqref{system} for nonlinear diffusion case. They represent scaling exponents, corresponding to the fact that, for each $\lambda>0$ and $m_1+m_2\neq m_1m_2$, the equations in \eqref{system} are invariant under the transformation $(u_1,u_2)\mapsto(u_{1\lambda},u_{2\lambda})$, where
$u_{1\lambda}(x,t)=\lambda^{d/p}u_1(\lambda x,\lambda^{d/q}t)$ and $u_{2\lambda}(x,t)=\lambda^{d/q}u_2(\lambda x,\lambda^{d/p}t)$. In \cite{CK2021-ANA}, we have obtained complete and optimal results for \eqref{system}: For $p<1$ or $q<1$ (we say that $\emph{\textbf{m}}:=(m_1,m_2)$ belongs to the sub-critical case), then weak solutions are global and uniformly bounded respect to time. For $p>1$ and $q>1$  (we say that $\emph{\textbf{m}}$ belongs to the super-critical case), blow-up solutions have been constructed for certain initial data. We also refer to \cite{L2023-CMS} for the global existence result with small initial data in the super-critical case. Obviously, $p=1$ or $q=1$ is the critical case for the solutions of \eqref{system}. Using an $(m_1,m_2)$ coordinate pair, one can show the critical cases are given by Line $L_1$ and Line $L_2$, reading as
\begin{equation*}\label{}
		\begin{split}
			\text{Line}\,\,\,L_1:\,\,m_1m_2+2m_1/d= m_1+m_2,\quad m_1\in\left[m^*,d/2\right),\quad m_2\in\left(1,m^*\right],
		\end{split}	
\end{equation*}
\begin{equation*}
		\begin{split}
			\text{Line}\,\,\,L_2:			m_1m_2+2m_2/d= m_1+m_2,\quad m_1\in\left(1,m^*\right],\quad m_2\in\left[m^*,d/2\right),
		\end{split}	
\end{equation*}
which intersects only at $(m^*,m^*)$. See the red curves in Figure \ref{figregimes}.
	
The purpose of this paper is to give a qualitative result on the asymptotic behavior of solutions in the area
	\begin{equation}\label{definition of m}
		\begin{split}
			m_1+m_2\geq &m_1m_2+2m_1/d,\,\,\,\,\,\,m_1+m_2\geq  m_1m_2+2m_2/d,\\
			m_1+m_2<&(1+2/d)m_1m_2,
		\end{split}
	\end{equation}
which is surrounded by Lines $L_i,i=1,2,3$, where the curve $L_3$ (blue curve in Figure \ref{figregimes}) is given by
	\begin{equation*}
		\begin{split}
			\text{Line}\,\,\,L_3:\,\,\, (1+2/d)m_1m_2= m_1+m_2\,\,\,\text{with}\,\,\, m_1\in\left(1,d/2\right),\,\,\,m_2\in\left(1,d/2\right).
		\end{split}	
	\end{equation*}
In particular, we would like to find a critical mass  phenomenon at the intersection point $(m^*,m^*)$. After rescaling, a special effort is  to find a balance between the diffusion and  the cross-attraction terms in \eqref{definition of free energy F}, this is related to the study of a maximization problem invoking the cross-attraction term in some suitable spaces (see Lemmas \ref{lemma maximizer for H 2} and \ref{lemma maximizer for H 4}). From these facts, the solutions will remain global over the time if the initial data ensures the positivity of free energy functional, while solutions satisfying certain constraints will lead to bounds on the total second moment implying blow up in finite time.

\
	
\noindent{\textbf{Main results.}} Let us introduce the concept of free energy solution for (\ref{system}).
	\begin{definition}\label{weak solution}
		Let $d\geq 3$ and $T>0$. Suppose that the initial data $\textbf{u}^0:=(u^0_1,u^0_2)$ satisfies
		\begin{equation} \label{total mass of solution}
			\begin{split}
				M_i:=\int_{\mathbb{R}^d}u_i(x,t)dx=\int_{\mathbb{R}^d}u^0_i(x)dx,\,\,i=1,2,\,\,\,t>0,
			\end{split}
		\end{equation}
		and
		\begin{equation} \label{assumption on initial data}
			\begin{split}
				&u^0_i\in L^1(\mathbb{R}^d;(1+|x|^2)dx)\cap L^{\infty}(\mathbb{R}^d),\\
				&u^{0}_i\in H^1(\mathbb{R}^d)\,\,\,\text{and}\,\,\,u^0_i\geq 0,\,\,i=1,2.
			\end{split}
		\end{equation}
		Then a vector $\textbf{u}:=(u_1,u_2)$ of non-negative functions defined in $\mathbb{R}^d\times(0,T)$ is called a free energy solution if
		\begin{enumerate}
			\item[$i)$]
			\begin{equation*}
				\begin{split}
					&u_i\in C([0,T);L^1(\mathbb{R}^d))\cap L^{\infty}(\mathbb{R}^d\times(0,T)),\\
					&u^{m_i}_i\in L^2(0,T;H^1(\mathbb{R}^d)),\quad i=1,2;
				\end{split}
			\end{equation*}
			\item[$ii)$] $\textbf{u}$ satisfies
			\begin{equation*}
				\begin{split}
					&\int^T_0\int_{\mathbb{R}^d}u_i(x,t)
					\partial_t{\phi}(x,t)dxdt+\int_{\mathbb{R}^d}u^0_i(x)\phi(x,0)dx\\
					&-\int^T_0\int_{\mathbb{R}^d}(\nabla u^{m_i}_i(x,t)-u_i(x,t)\nabla v_j(x,t))\cdot \nabla \phi(x,t) dxdt=0,
				\end{split}
			\end{equation*}
			for any test function $\phi\in \mathcal{D}(\mathbb{R}^d\times[0,T))$, $\forall\,\,i,j=1,2,\,\,i\neq j$;
			\item[$iii)$]	for any $T>0$, it satisfies 
			\begin{center}
				$u^{(2m_i-1)/2}_i\in L^2(0,T;H^1(\mathbb{R}^d))$, $i=1,2$, 
			\end{center}
			and free energy inequality
			\be{\label{free energy inequality}}
			\begin{split}	
				\mathcal{F}[u_1,u_2](t)+\int^t_ {0}\mathcal{D}[u_1,u_2](s)ds
				\leq \mathcal{F}[u^0_1,u^0_2]\quad\text{for}\quad t\in(0,T).
			\end{split}
			\ee	
		\end{enumerate}
	\end{definition}
		
The local-in-time existence of free energy solution of \eqref{system} has been guaranteed by the limit of strong solution to the corresponding approximated system. According to the analytic semigroup theory \cite[Theorem IV.1.5.1]{A1995} and the fixed point theorem, there exist  $T_{\max}\in(0,\infty]$ and a free energy solution $\emph{\textbf{u}}$ on $[0,T_{\max})$ with the following
	alternative: either $T_{\max}=\infty$ or if $T_{\max}<\infty$, then
	\begin{equation*}\label{criterion on local solution}
		\sum\limits^2_{i=1}||u_i(\cdot,t)||_{\infty}\rightarrow\infty \,\,\,\text{as}\,\,t\rightarrow T_{\max}.
	\end{equation*}
See \cite[Propositions 8-10]{Sugiyama2006-JDE} for the proof of one-species chemotaxis system, and see \cite[Section 2]{CK2021-ANA} for \eqref{system}.
Denote 
	\begin{center}$r:=(1+2/d)m_1m_2/(m_1+m_2)>1$ with $m_1,m_2\in(1,d/2)$. 
	\end{center}
It should be noted that the region $\emph{\textbf{m}}=(m_1,m_2)$ of exponents determined by the conditions in \eqref{definition of m} is equivalent to the conditions $p\geq 1$, $q\geq 1$ and  $r>1$. Define
	\begin{equation*}{\label{definition of set of minimizer 2}}
		\begin{split}
			\Gamma_i=\Big\{&h_i|h_i\in L^{1}(\mathbb{R}^d)\cap  L^{m_i}(\mathbb{R}^d), h_i\geq 0\Big\},\quad i=1,2.
		\end{split}
	\end{equation*}	
Chosen $\alpha,\beta>0$ satisfying
	\begin{equation*}{\label{assumption on alpha}}
		\begin{split}
			\alpha\in\left(0,\frac{m_1}{m_1-1}\left(1+\frac{2}{d}-\frac{1}{m_1}-\frac{1}{m_2}\right)\right),
		\end{split}	
	\end{equation*}
	\begin{equation*}{\label{}}
		\begin{split}
			\beta\in\left(0,\frac{m_2}{m_2-1}\left(1+\frac{2}{d}-\frac{1}{m_1}-\frac{1}{m_2}\right)\right),
		\end{split}	
	\end{equation*}
and
	\begin{equation*}{\label{relationship between alpha and beta 2}}
		\begin{split}
			\frac{m_1-1}{m_1}\alpha+\frac{m_2-1}{m_2}\beta=1+\frac{2}{d}-\frac{1}{m_1}-\frac{1}{m_2}.
		\end{split}	
	\end{equation*}
Then $(1-\alpha)/m_1+(1-\beta)/m_2>1$ if $m_1,m_2\neq m^*$, and $(1-\alpha)/m_1+(1-\beta)/m_2=1$ if both $m_1$ and $m_2$ are  equal to $m^*$. Furthermore, the variational problem
	\begin{equation*}{\label{definiton of C}}
		\begin{split}
			\Lambda^*_{m_1,m_2}=\sup_{0\neq h_i\in \Gamma_i,\,\,i=1,2}\frac{\mathcal{H}[h_1,h_2]}{\|h_1\|^{\alpha}_{1}\|h_2\|^{\beta}_{1}\|h_1\|^{1-\alpha}_{m_1}\|h_2\|^{1-\beta}_{m_2}}
		\end{split}
	\end{equation*}
is well-defined (see Section \ref{Variants of Hardy-Littlewood-Sobolev inequalities}).
	
To state our results precisely, we will introduce some notations as follows. Let $m_1,m_2\neq m^*$, and let $\theta\in(0,1)$. Define $\gamma_1,\gamma_2>0$ as
\begin{equation*}\label{definition of gamma 1 2}
			\gamma_1=\frac{\alpha}{	\frac{1-\alpha}{m_1}+\frac{1-\beta}{m_2}-1},\quad 	\gamma_2=\frac{\beta}{	 \frac{1-\alpha}{m_1}+\frac{1-\beta}{m_2}-1},
	\end{equation*}	
and set
\begin{equation*}\label{definition of kamma 1 2}
		\begin{split}
			\kappa_1=&c_d^{\frac{m_2}{m_1+m_2-m_1m_2}}\left[(m_1-1)\theta\right]^{\frac{m_2-1}{m_1+m_2-m_1m_2}}\left[(m_2-1)(1-\theta)\right]^{\frac{1}{m_1+m_2-m_1m_2}},\\
			\kappa_2=&c_d^{\frac{m_1}{m_1+m_2-m_1m_2}}\left[(m_1-1)\theta\right]^{\frac{1}{m_1+m_2-m_1m_2}}\left[(m_2-1)(1-\theta)\right]^{\frac{m_1-1}{m_1+m_2-m_1m_2}}.
		\end{split}
	\end{equation*}
Define an auxiliary function
$$
f(x):=x-\Lambda^{*}_{m_1,m_2,\theta}x^{\frac{1-\alpha}{m_1}+\frac{1-\beta}{m_2}}
$$
for $x>0$, where  
$$
\Lambda^{*}_{m_1,m_2,\theta}:=A^{\frac{1-\alpha}{m_1}+\frac{1-\beta}{m_2}}\left(\theta,\frac{1-\alpha}{m_1}/\left[\frac{1-\alpha}{m_1}+\frac{1-\beta}{m_2}\right]\right)\Lambda^*_{m_1,m_2}>0
$$ 
with $A(\theta,\eta)>0$ for $\theta,\eta\in(0,1)$ given in Lemma \ref{lemma for Young inequality}. It is evident that the function $f(x)$ reaches its maximum value at
	\begin{equation}\label{definition of x0}
			x_0=\left[\frac{1}{\left(\frac{1-\alpha}{m_1}+\frac{1-\beta}{m_2}\right)\Lambda^{*}_{m_1,m_2,\theta}}\right]^{\frac{1}{\frac{1-\alpha}{m_1}+\frac{1-\beta}{m_2}-1}}.
	\end{equation}

Note that both $\gamma_1,\gamma_2$ makes no sense and the maximum point $x_0$ of $f$ does not exist when $m_1=m_2=m^*$. Hence we first provide the main arguments about the behavior of free energy solutions to \eqref{system} when $m_1,m_2>1$ satisfy \eqref{definition of m}, excluding the point of intersection $(m^*,m^*)$. 	
	\begin{theorem}{\label{theorem on global existence}}
		Let $d\geq 3, T_{\max}\in(0,\infty]$. Let  $p\geq 1$, $q\geq 1$, $r>1$, and let $m_1,m_2\neq m^*$. Suppose that there exists a local free energy solution $\textbf{u}$ of \eqref{system} over $[0,T_{\max})$ with initial data $\textbf{u}^0$ satisfying \eqref{total mass of solution} and \eqref{assumption on initial data}.  Assume that there exists $\theta\in (0,1)$ such that	\begin{equation*}\label{assumption on F with initial data}
			\begin{split}
				\mathcal{F}[u^0_1,u^0_2]< \frac{c_d}{\kappa^{1+\gamma_1}_1\kappa^{1+\gamma_2}_2M^{\gamma_1}_1M^{\gamma_2}_2}f(x_0).
			\end{split}
		\end{equation*}	
  with $x_0$ depending on $\theta$ given by \eqref{definition of x0}. Define
		\begin{enumerate}
			\item[(i)] If
$			\mathcal{R}:=\kappa^{\gamma_1}_1\kappa^{\gamma_2}_2M^{\gamma_1}_1M^{\gamma_2}_2\left[\theta\kappa^{m_1}_1\|u^0_1\|^{m_1}_{m_1}+(1-\theta)\kappa^{m_2}_2\|u^0_2\|^{m_2}_{m_2}\right]<x_0$,
			then the corresponding free energy solution exists globally.
			\item[(ii)] If
			\begin{equation}\label{assumption on m1 and m2 1}
				\begin{split}
					\max\{m_1,m_2\}\leq2-\frac{2}{d}-\frac{d-2}{d}\left(1-\frac{1}{\frac{1-\alpha}{m_1}+\frac{1-\beta}{m_2}}\right)
				\end{split}
			\end{equation}
			and $\mathcal{R}>x_0$,
			then the free energy solution does not exist globally in time.
		\end{enumerate}
Moreover, the critical case $\mathcal{R}=x_0$ does not depend on $\theta$.
	\end{theorem}
 
 \begin{remark}
The above theorem provides a classification of the occurrence of finite time blow-up and global existence of free energy solutions in the region \eqref{definition of m} with $m_1,m_2\neq m^*$. This classification is based on the combination of invariant norms of two distinct initial data given by the critical identity $\mathcal{R}=x_0$. This critical identity can be simplified to
		\begin{equation*}\label{assumption on M1 and M2 4}
	\begin{split}
	c^{\frac{1}{\frac{1-\alpha}{m_1}+\frac{1-\beta}{m_2}-1}}_d&M^{\gamma_1}_1M^{\gamma_2}_2\Big[(m_1-1)^{\frac{1-\frac{1-\beta}{m_2}}{\frac{1-\alpha}{m_1}+\frac{1-\beta}{m_2}-1}}(m_2-1)^{\frac{\frac{1-\beta}{m_2}}{\frac{1-\alpha}{m_1}+\frac{1-\beta}{m_2}-1}}\|u^0_1\|^{m_1}_{m_1}\\
	&+(m_1-1)^{\frac{\frac{1-\alpha}{m_1}}{\frac{1-\alpha}{m_1}+\frac{1-\beta}{m_2}-1}}(m_2-1)^{\frac{1-\frac{1-\alpha}{m_1}}{\frac{1-\alpha}{m_1}+\frac{1-\beta}{m_2}-1}}\|u^0_2\|^{m_2}_{m_2}\Big]\\
	=&\left(\frac{1-\alpha}{m_1}+\frac{1-\beta}{m_2}\right)\left[\frac{1}{\left(\frac{1-\alpha}{m_1}\right)^{\frac{1-\alpha}{m_1}}\left(\frac{1-\beta}{m_2}\right)^{\frac{1-\beta}{m_2}}\Lambda^*_{m_1,m_2}}\right]^{\frac{1}{\frac{1-\alpha}{m_1}+\frac{1-\beta}{m_2}-1}},
	\end{split}
\end{equation*}
that is independent of the parameter $\theta\in(0,1)$. The restriction on the set of exponents $\emph{\textbf{m}}$ given by \eqref{assumption on m1 and m2 1} on Theorem \ref{theorem on global existence} (ii) to obtain blow-up for initial data satisfying $\mathcal{R}>x_0$ is not sharp, see more details in Section 5.
\end{remark}

The second theorem demonstrates a dichotomy phenomenon when $\emph{\textbf{m}}$ is at the intersection point. Let  $M_1$ and $M_2$ be positive numbers, and let $\theta_0=M^{m^*}_1/(M^{m^*}_1+M^{m^*}_2)\in(0,1)$.	Denote
\begin{equation}\label{definition of Pi}
			\Pi^*_{\theta_0}:=\sup_{0\neq h_i\in\Gamma_i,i=1,2}\frac{\mathcal{H}[h_1,h_2]}{\|h_1\|_1\|h_2\|_1[\theta_0\| h_1\|^{-m^*}_1\|h_1\|^{m^*}_{m^*}+(1-\theta_0)\|h_2\|^{-m^*}_1\|h_2\|^{m^*}_{m^*}]},
\end{equation}
and  define
$$
		\Sigma(\boldsymbol{M})=c_d(m^*-1)\Pi^*_{\theta_0}M_1M_2/(M^{m^*}_1+M^{m^*}_2).
$$
We have the following result.
\begin{theorem}{\label{theorem on critical on the intersection point}}
Let $m_1=m_2=m^*$, and let $M_1,M_2>0$. Suppose that there exists a local free energy solution $\textbf{u}$ of \eqref{system} over $[0,T_{\max})$ with initial data $\textbf{u}^0$ satisfying \eqref{total mass of solution} and \eqref{assumption on initial data}. Then
	\begin{enumerate}
		\item[(i)] If $\Sigma(\boldsymbol{M})<1$, then the corresponding free energy solution exists globally.
		\item[(ii)] If $\Sigma(\boldsymbol{M})>1$, 	then the free energy solution does not exist globally in time.
	\end{enumerate}
\end{theorem}	

Notice that we can find lower and upper bounds for $\Pi^*_{\theta_0}$. Define
\begin{equation}{\label{definition of Cstar}}
		C_*:=\sup_{h\neq 0}\left\{\frac{\mathcal{H}[h,h]}{\|h\|^{2/d}_1\|h\|^{m^*}_{m^*}},h\in L^1(\mathbb{R}^d)\cap L^{m^*}(\mathbb{R}^d)\right\}.
\end{equation}
As a simple observation, we  conclude that $C_*\leq \Pi^*_{\theta_0}$ for $\theta_0\in(0,1)$. In fact, \cite[Lemma 3.3]{BCL09-CVPDE} guarantees the existence of a non-negative function $h^*\in L^1(\mathbb{R}^d)\cap L^{m^*}(\mathbb{R}^d)$ with $\|h^*\|_1=\|h^*\|_{m^*}=1$ such that $\mathcal{H}[h^*,h^*]=C_*$, and hence $C_*\leq \Pi^*_{\theta_0}$ by setting $h_1=h_2=h^*$ in \eqref{definition of Pi}. In particular, our Theorem \ref{theorem on critical on the intersection point} $(ii)$ implies that if 
\begin{equation}\label{assumption on global existence condition 1}
    c_d(m^*-1)C_*M_1M_2/(M^{m^*}_1+M^{m^*}_2)>1,
\end{equation}
then there exists initial data leading to blow-up.

On the other hand, we shall show in Remark \ref{remark for pi and lambda} that 
\begin{equation}{\label{inequality for Pi theta0}}	\Pi^*_{\theta_0}\leq\frac{M_c(M^{m^*}_1+M^{m^*}_2)}{c_d(m^*-1)M^{1-\alpha}_1M^{1-\beta}_2},
\end{equation}
where $M_c$ is given by
\begin{equation}{\label{definition of Mc}}
	M_c=c_d\Lambda^*_{m^*,m^*}\frac{m^*-1}{m^*}\left(1-\alpha\right)^{\frac{1-\alpha}{m^*}}\left(1-\beta\right)^{\frac{1-\beta}{m^*}},
\end{equation}
which satisfies
\begin{equation}{\label{inequality for Mc}}
	M_c=c_d\Lambda^*_{m^*,m^*}(m^*-1)z\left(\frac{1-\alpha}{m^*}\right)\geq\frac{c_d\Lambda^*_{m^*,m^*}(m^*-1)}{2}\geq \frac{c_dC_*(m^*-1)}{2}
\end{equation}
with $z(\eta)=\eta^{\eta}\left(1-\eta\right)^{1-\eta}$, $\eta\in(0,1)$. The first inequality is a result of the fact that the function $z$ achieves its minimal value only at $\eta=1/2$.  The second one is a consequence of a simple fact $C_*\leq \Lambda^*_{m^*,m^*}$. By \eqref{inequality for Pi theta0} and \eqref{inequality for Mc}, we have $\Sigma(\boldsymbol{M})\leq M^{\alpha}_1M^{\beta}_2M_c$. In particular, our Theorem \ref{theorem on critical on the intersection point} $(i)$ implies that if
\begin{equation}\label{assumption on blow-up condition 2}
M^{\alpha}_1M^{\beta}_2<M^{-1}_c,
\end{equation}
the solution is global. 

The global existence condition \eqref{assumption on blow-up condition 2} and blow-up condition \eqref{assumption on global existence condition 1} with $\alpha=\beta=1/d$ were previously derived in \cite[Theorem 1.4]{CK2021-ANA}.
We established in \cite[Theorem 1.4]{CK2021-ANA} that solutions exist globally below the curve $M^{1/d}_1M^{1/d}_2=M^{-1}_c$, represented by the cyan curve in Figure \ref{figreg}, and blow-up occurs above the curve $M_1M_2/(M^{m_c}_1+M^{m_c}_2)=1/(2M_c)$, represented as the green curve in Figure \ref{figreg}. Our main Theorem \ref{theorem on critical on the intersection point} shows that the sharp curve, represented as the purple curve in Figure \ref{figreg} and resulting in a dichotomy at the intersection point ${\bf{I}}=(m^*,m^*)$, is given by the condition $\Sigma(\boldsymbol{M})=1$, fully solving the open problem in \cite{CK2021-ANA}. 
  
	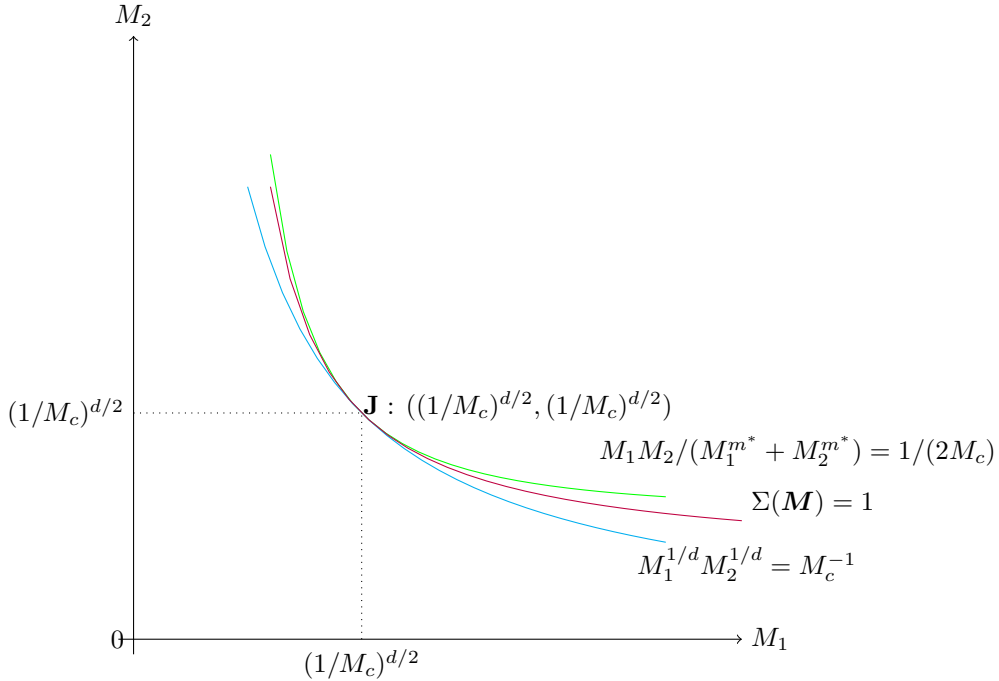
\begin{figure}[ht!]
		\centering
		\begin{tikzpicture}
			\draw (0,0) node[left] {$0$};
			\draw[->] (-0.2,0) --(8,0) node[right] {$M_1$};
			\draw[->] (0,-0.2) --(0,8) node[above] {$M_2$};
			\draw[dotted] (3,0)--(3,3) node[below] at(3,0) {$(1/M_c)^{d/2}$};
			\draw[dotted] (0,3)--(3,3) node[left] at(0,3) {$(1/M_c)^{d/2}$};	
			\draw[domain =1.8:7, green] plot (\x ,{5/3+32/3 *1/(\x^2-1)});
			\node[right] at(6,2.5) {$M_1M_2/(M^{m^*}_1+M^{m^*}_2)=1/(2M_c)$};
			\draw[domain =1.5:7,cyan] plot (\x ,{9/\x});
      			\node[right] at(8,1.8) {$\Sigma(\boldsymbol{M})=1$};
   			\draw[domain =1.8:8,purple] plot (\x ,{1+4/(\x-1)});
			\draw (2.9,3.1) node[right]{${\bf{J}}:\,((1/M_c)^{d/2},(1/M_c)^{d/2})$};
			\node[right] at(6.5,1) {$M^{1/d}_1M^{1/d}_2=M^{-1}_c$};
		\end{tikzpicture}
		\caption{Parameter line on intersection point $\bf{I}$.}
		\label{figreg}
	\end{figure}

Let us finally observe that if $h_1=h_2$ in \eqref{definition of Pi} then $\Pi^*_{\theta_0}=C_*$. Consequently,
Theorem \ref{theorem on critical on the intersection point} can be applied to the one-single chemotaxis system \eqref{KS system} recovering the well-known critical mass condition $\|u^0\|_1=M=\left(2/[c_dC_*(m^*-1)]\right)^{d/2}$ for the dichotomy, global existence versus blow-up, as in \cite{BCL09-CVPDE}.

This paper is organized as follows. In Section \ref{Variants of Hardy-Littlewood-Sobolev inequalities}, we derive the optimal constants for two modified Hardy-Littlewood-Sobolev inequalities in higher dimensions. Section \ref{Sufficient condition on well-posedness of solutions} provides a sufficient condition on the global existence of free energy solution to \eqref{system}. The final two Sections \ref{global existence} and \ref{blow-up} are devoted to the proof of well-posedness and blow-up of solutions to \eqref{system} in Theorems \ref{theorem on global existence} and \ref{theorem on critical on the intersection point}, respectively.

\section{Variants of Hardy-Littlewood-Sobolev inequalities}\label{Variants of Hardy-Littlewood-Sobolev inequalities}
Let us first recall the following well-known Hardy-Littlewood-Sobolev inequality (see \cite[Theorem 4.3]{Lieb2001}). Given $1<p_i<\infty$, $i=1,2$, and $1/p_1+1/p_2=1+2/d$. For all $h_i\in L^{p_i}(\mathbb{R}^d)$, there exists a constant $C_{\text{HLS}}=C_{\text{HLS}}(p_1,p_2,d)>0$ such that
	\begin{equation}\label{HLS inequality}
		\begin{split}
			|\mathcal{H}[h_1,h_2]|=\left|\iint_{\mathbb{R}^d\times\mathbb{R}^d}\frac{h_1(x)h_2(y)}{|x-y|^{d-2}}dxdy\right|\leq C_{\text{HLS}}\|h_1\|_{p_1}\|h_2\|_{p_2}.
		\end{split}
	\end{equation}
The sharp constant
	$$
	C_{\text{HLS}}=\sup_{h_i\neq 0,i=1,2}\left\{\frac{\mathcal{H}[h_1,h_2]}{\|h_1\|_{p_1}\|h_2\|_{p_2}}, h_i\in L^{p_i}(\mathbb{R}^d), i=1,2\right\}
	$$
was given by Lieb in \cite{Lieb1983-AM}, where it was shown that a maximizing pair $h_1$, $h_2$ exists for \eqref{HLS inequality}, i.e., a pair that gives equality in \eqref{HLS inequality}.
	
Next, we present two modifications to the Hardy-Littlewood-Sobolev inequality that are essential for establishing the global well-posedness of solutons. Let $\theta\in(0,1)$, $p\geq 1$, $q\geq 1$ and $r>1$. Consider the following problem
	\begin{equation}{\label{definiton of Lambda}}
		\begin{split}
			\Lambda^*_{m_1,m_2}=\sup_{0\neq h_i\in \Gamma_i,\,\,i=1,2}\frac{\mathcal{H}[h_1,h_2]}{\|h_1\|^{\alpha}_{1}\|h_2\|^{\beta}_{1}\|h_1\|^{1-\alpha}_{m_1}\|h_2\|^{1-\beta}_{m_2}},
		\end{split}
	\end{equation}
	where 
	\begin{equation}{\label{definition of alpha}}
		\begin{split}
			\alpha\in\left(0,\frac{m_1}{m_1-1}\left(1+\frac{2}{d}-\frac{1}{m_1}-\frac{1}{m_2}\right)\right),
		\end{split}	
	\end{equation}
	\begin{equation}{\label{definition of beta}}
		\begin{split}
			\beta\in\left(0,\frac{m_2}{m_2-1}\left(1+\frac{2}{d}-\frac{1}{m_1}-\frac{1}{m_2}\right)\right),
		\end{split}	
	\end{equation}
	are assumed to satisfy the following condition
	\begin{equation}{\label{definition of alpha and beta}}
		\begin{split}
			\frac{m_1-1}{m_1}\alpha+\frac{m_2-1}{m_2}\beta=1+\frac{2}{d}-\frac{1}{m_1}-\frac{1}{m_2}.
		\end{split}	
	\end{equation}
Note that  $(1-\alpha)/m_1+(1-\beta)/m_2\geq 1$. Assume, without loss of generality, that $m_1\geq m_2$. We can then notice the following inequality for $\alpha$ and $\beta$:
		\begin{equation}\label{inequality for alpha and beta}
			\begin{split}
				\alpha+\beta=&\left(1+\frac{2}{d}-\frac{1}{m_1}-\frac{1}{m_2}-\frac{m_2-1}{m_2}\beta\right)\frac{m_1}{m_1-1}+\beta\\
				=&\left(1+\frac{2}{d}-\frac{1}{m_1}-\frac{1}{m_2}\right)\frac{m_1}{m_1-1}+\frac{m_1-m_2}{(m_1-1)m_2}\beta\\
				\leq&\left(1+\frac{2}{d}-\frac{1}{m_1}-\frac{1}{m_2}\right)\left[\frac{m_1}{m_1-1}+\frac{m_1-m_2}{(m_1-1)(m_2-1)}\right]\\
				=&\frac{m_2}{m_2-1}\left(1+\frac{2}{d}-\frac{1}{m_1}-\frac{1}{m_2}\right)\\
				\leq&\frac{2}{d}.
			\end{split}
		\end{equation}
This implies that
		\begin{equation*}\label{inequality for alpha and beta 2}
			\begin{split}
				\frac{1-\alpha}{m_1}+\frac{1-\beta}{m_2}=1+\frac{2}{d}-\alpha-\beta\geq 1.
			\end{split}
		\end{equation*}
Furthermore, it is apparent from the inequality shown in \eqref{inequality for alpha and beta} that $\alpha+\beta=2/d$ if and only if $m_1=m_2=m^*$.
	
Finally, the definition of $\Lambda^*_{m_1,m_2}>0$ in \eqref{definiton of Lambda} is valid. By utilizing the  Hardy-Littlewood-Sobolev inequality with $r_1\in[m_2/[(1+2/d)m_2-1],m_1]$  and $r_2=r_1/[(1+2/d)r_1-1]\in(1,m_2]$, wee see that  for $h_i\in \Gamma_i$, $i=1,2$,
	\begin{equation*}\label{inequality for H}
		\begin{split}
			\mathcal{H}[h_1,h_2]\leq&C_{\text{HLS}}\|h_1\|_{r_1}\|h_2\|_{\frac{r_1}{(1+2/d)r_1-1}}\\
			\leq&C_{\text{HLS}}\|h_1\|^{\alpha}_{1}\|h_1\|^{1-\alpha}_{m_1}\|h_2\|^{\beta}_{1}\|h_2\|^{1-\beta}_{m_2}\\
			=&C_{\text{HLS}}\|h_1\|^{\alpha}_{1}\|h_2\|^{\beta}_{1}\|h_1\|^{1-\alpha}_{m_1}\|h_2\|^{1-\beta}_{m_2}.
		\end{split}
	\end{equation*}	
Hence $\Lambda^*_{m_1,m_2}\leq C_{\text{HLS}}$. Before demonstrating the existence of a maximizer for $\Lambda^*_{m_1,m_2}$, let us provide the following lemma \cite[Lemma 2.4]{Lieb1983-AM}.
		\begin{lemma}{\label{lemma in Lieb paper}}
			Let $\epsilon\in(0,1)$, and let $1/r_1=1/r_2+2/d$ with $1<r_1<r_2$. Suppose that non-negative $ \vartheta\in L^{r_1}(\mathbb{R}^d)$ is spherically symmetric and $\vartheta(\rho)\leq\epsilon \rho^{-d/r_1}$ for all $\rho>0$.  Then there exists a constant $C=C(d)>0$ such that
			\begin{equation*}	{\label{}}
				\begin{split}	
					\||x|^{-(d-2)}\ast\vartheta\|_{r_2}\leq C\epsilon^{1-r_1/r_2}\|\vartheta\|^{r_1/r_2}_{r_1}.
				\end{split}
			\end{equation*}
		\end{lemma}
Based on the arguments presented in \cite[Lemma 3.3]{BCL09-CVPDE}, which focuses on the maximization problem \eqref{definiton of Lambda} under the specific conditions of   $h_1=h_2$, $m_1=m_2=m^*$ and $\alpha+\beta=2/d$, we establish the existence of a maximizer for $\Lambda^*_{m_1,m_2}$ as follows.
	\begin{lemma}{\label{lemma maximizer for H 2}}
		Let $\alpha,\beta>0$ satisfy \eqref{definition of alpha}-\eqref{definition of alpha and beta}. Then there exists a pair $(\phi^*_1,\phi^*_2)\in \Gamma_1\times\Gamma_2$ of non-negative, radially symmetric and non-increasing function satisfying
$			\|\phi^*_{1}\|^{\alpha}_1\|\phi^*_{2}\|^{\beta}_1=1
$
and
$	\|\phi^*_{1}\|_{m_1}=\|\phi^*_{2}\|_{m_2}=1$. This pair of functions achieves the value of $\Lambda^*_{m_1,m_2}$.
	\end{lemma}
\begin{proof}
Suppose that there exists a maximum  sequence $(h_{ij})_j$ with non-negative elements in $\Gamma_i$, $i=1,2$, such that
		\begin{equation*}	{\label{minimizing sequence for F}}
				\lim_{j\rightarrow\infty}\mathcal{W}[h_{1j},h_{2j}]=\Lambda^*_{m_1,m_2}.
		\end{equation*}	
The function $\mathcal{W}$ is defined by
		\begin{equation*}	{\label{}}
				\mathcal{W}[h_{1},h_{2}]:=\frac{\mathcal{H}[h_1,h_2]}{\|h_1\|^{\alpha}_1\|h_2\|^{\beta}_1\|h_1\|^{1-\alpha}_{m_1}\|h_2\|^{1-\beta}_{m_2}},\quad h_i\in\Gamma_i,\quad i=1,2.
		\end{equation*}	
Denote $\phi_{ij}(x):=\lambda_{ij}h_{ij}(\mu_j x)$, $i=1,2$,  $j\in\mathbb{N}^{+}$, where
		\begin{equation*}{\label{}}
				\lambda_{ij}=\mu^{\frac{d}{m_i}}_j\|h_{ij}\|^{-1}_{m_i},\quad i=1,2,
		\end{equation*}
		\begin{equation*}{\label{}}
			\begin{split}
				\mu_j=&\|h_{1j}\|^{\frac{\alpha}{d\left(\frac{m_1-1}{m_1}\alpha+\frac{m_2-1}{m_2}\beta\right)}}_1\|h_{2j}\|^{\frac{\beta}{d\left(\frac{m_1-1}{m_1}\alpha+\frac{m_2-1}{m_2}\beta\right)}}_1\\
				&\cdot\|h_{1j}\|^{-\frac{\alpha}{d\left(\frac{m_1-1}{m_1}\alpha+\frac{m_2-1}{m_2}\beta\right)}}_{m_1}\|h_{2j}\|^{-\frac{\beta}{d\left(\frac{m_1-1}{m_1}\alpha+\frac{m_2-1}{m_2}\beta\right)}}_{m_2}.
			\end{split}
		\end{equation*}
Upon a direct computation, it is evident that
		\begin{equation*}{\label{}}
				\mathcal{W}[h_{1j},h_{2j}]=\mathcal{W}[\phi_{1j},\phi_{2j}]=\mathcal{H}[\phi_{1j},\phi_{2j}],
		\end{equation*}
		\begin{equation*}	{\label{}}	
				\|\phi_{1j}\|^{\alpha}_1\|\phi_{2j}\|^{\beta}_1=1,
		\end{equation*}
		\begin{equation*}{\label{}}	
				\|\phi_{1j}\|_{m_1}=\|\phi_{2j}\|_{m_2}=1.
		\end{equation*}
Define  $\phi^*_{ij}$ as the symmetric non-increasing rearrangement of $\phi_{ij}$. By referring to the Riesz rearrangement property in \cite[Lemma 2.1]{Lieb1983-AM}, one obtains that
		\begin{equation}	{\label{equality for psi 1j and psi 2j}}	
				\|\phi^*_{1j}\|^{\alpha}_1\|\phi^*_{2j}\|^{\beta}_1=1,
		\end{equation}
		\begin{equation}	{\label{equality for psi 1j and psi 2j 2}}	
				\|\phi^*_{1j}\|_{m_1}=\|\phi^*_{2j}\|_{m_2}=1,
		\end{equation}
		\begin{equation*}	
				\mathcal{H}[\phi^*_{1j},\phi^*_{2j}]=\mathcal{W}[\phi^*_{1j},\phi^*_{2j}]\geq\mathcal{W}[\phi_{1j},\phi_{2j}].
		\end{equation*}	
Hence $(\phi^*_{ij})_j$ is also a maximizing sequence. Using \eqref{equality for psi 1j and psi 2j 2} and the fact that $\phi^*_{ij}$ is non-negative and monotonically increasing, we conclude that for $R>0$,
		\begin{equation*}	
			\begin{split}
				1=\|\phi^*_{ij}\|^{m_i}_{m_i}=&c_d\int^{\infty}_0(\phi^*_{ij})^{m_i}(\rho)\rho^{d-1}d\rho\geq c_d\int^{R}_0(\phi^*_{ij})^{m_i}(\rho)\rho^{d-1}d\rho\\
				\geq&\frac{c_dR^d}{d}(\phi^*_{ij})^{m_i}(R),\,\,i=1,2.
			\end{split}
		\end{equation*}	
Then
		\begin{equation}\label{inequality for phi}	
				0\leq \phi^*_{ij}(R)\leq \left(\frac{d}{c_d}\right)^{\frac{1}{m_i}}R^{-d/m_i}\leq C_0R^{-d/m_i},\,\,i=1,2.
		\end{equation}	
Therefore based on the boundedness and decreasing properties of $\phi^*_{ij}$ in $(R,\infty)$, $i=1,2$, one makes use of Helly's theorem to find a radially symmetric and non-increasing subsequence fulfilling
		\begin{equation}{\label{weak convergence for f*}}
				\phi^*_{ij}\rightarrow \phi^*_i,\,\,\,\,\text{point-wisely}\,\,\,\,\text{as}\,\,\,\,j\rightarrow\infty,\quad i=1,2,
		\end{equation}	
with some non-negative symmetric non-increasing $\phi^*_i\in L^{m_i}(\mathbb{R}^d)$. Furthermore, the inequality
		\begin{equation}{\label{lower bound for w*}}
				\|\phi^*_{i}\|_{m_i}\leq1,\quad  i=1,2,
		\end{equation}	
is valid due to Fatou's lemma.	
		
In the subsequent step, we assert that
\begin{center}$\|\phi^*_i\|_1\neq 0$, $\forall\,\,i=1,2$. 
		\end{center}
This implies that $\|\phi^*_{ij}\|_1$ is uniformly bounded for all $j\in\mathbb{N}^+$ by \eqref{equality for psi 1j and psi 2j},  $\forall\,\,i=1,2$. Otherwise, if $\|\phi^*_{ij}\|_1\rightarrow 0$ for some $i\in\{1,2\}$ as $j\rightarrow\infty$, then $ \|\phi^*_{i}\|_1\equiv 0$ according to Fatou's lemma, which contradicts the previous statement.  
		
To demonstrate above claim, without loss of generality we first assume $\|\phi^*_{1j}\|_1\leq 1$ for all $j\in \mathbb{N}^+$ based on \eqref{equality for psi 1j and psi 2j}. Let $r_1=m_2/[(1+2/d)m_2-1]\in(1,m_1)$ and $r_2=m_2/(m_2-1)$. The monotony of  $\phi^*_{1j}$ implies that	
\begin{equation*}	
				(\phi^*_{1j})^{r_1}(R)\leq\frac{d}{c_dR^d} \|\phi^*_{1j}\|^{r_1}_{r_1}\leq\frac{d}{c_dR^d} \|\phi^*_{1j}\|^{\frac{m_1-r_1}{m_1-1}}_{1}\|\phi^*_{1j}\|^{\frac{(r_1-1)m_1}{m_1-1}}_{m_1}\leq \frac{d}{c_d}R^{-d}.
		\end{equation*}
Denote $a_{1j}=\sup_{\rho}\rho^{d/r_1}\phi^*_{1j}(\rho)$. If $a_{1j}\rightarrow 0$, then Lemma \ref{lemma in Lieb paper} implies that $\||x|^{-(d-2)}\ast\phi^*_{1j}\|_{r_2}\rightarrow 0$, whereas H\"{o}lder's inequality guarantees that
		\begin{equation*}{\label{}}
				\left|\mathcal{H}[\phi^*_{1j},\phi^*_{2j}]\right|\leq \||x|^{-(d-2)}\ast\phi^*_{1j}\|_{r_2}\|\phi^*_{2j}\|_{m_2}\rightarrow 0.
		\end{equation*}		
This is impossible since $(\phi^*_{ij})_j$, $i=1,2$, is a maximizing sequence. Then there exist constants $\delta>0$ and $\rho_0>0$ such that $a_{1j}\geq \delta>0$ and  $\phi^*_{1j}(\rho_0)>\delta\rho^{-d/r_1}_0/2>0$. Therefore, it follows that $\phi^*_{1}(\rho)>\delta\rho^{-d/r_1}_0/2$ for all $\rho\leq \rho_0$, and then $\|\phi^*_1\|_1>c_d\delta\rho^{d-d/r_1}_0/(2d)$.
		
We next show $\|\phi^*_2\|_1\neq 0$. By taking $r_1=m_2$ and $r_2=m_2/(1-2m_2/d)$ in Lemma \ref{lemma in Lieb paper}, we obtain $(\phi^*_{2j})^{r_1}(R)\leq d/(c_dR^d)$. Define $a_{2j}=\sup_{\rho}\rho^{d/r_1}\phi^*_{2j}(\rho)$. Using a comparable reasoning, it can be deduced that when $a_{2j}\rightarrow 0$,  $\||x|^{-(d-2)}\ast\phi^*_{2j}\|_{r_2}\rightarrow 0$. Thus H\"{o}lder's inequality implies that
		\begin{equation*}{\label{}}
			\begin{split}
				\left|\mathcal{H}[\phi^*_{1j},\phi^*_{2j}]\right|&\leq \|\phi^*_{1j}\|_{\frac{r_2}{r_2-1}}\||x|^{-(d-2)}\ast\phi^*_{2j}\|_{r_2}\\
				&=\|\phi^*_{1j}\|_{\frac{m_2}{(1+2/d)m_2-1}}\||x|^{-(d-2)}\ast\phi^*_{2j}\|_{r_2}\\
				&\leq\|\phi^*_{1j}\|^{\frac{(1+2/d)m_1m_2-m_1-m_2}{(m_1-1)m_2}}_{1}\|\phi^*_{1j}\|^{\frac{m_1(1-2/dm_2)}{(m_1-1)m_2}}_{m_1}\||x|^{-(d-2)}\ast\phi^*_{2j}\|_{r_2}\\
				&\leq\||x|^{-(d-2)}\ast\phi^*_{2j}\|_{r_2}
				\rightarrow 0.
			\end{split}
		\end{equation*}	
This also contradicts the selection of the maximizing sequence, and  so $\|\phi^*_2\|_1\neq 0$.
		
We now have a maximizing sequence $(\phi^*_{ij})_j$ of  non-negative symmetric decreasing functions which converge pointwise everywhere to non-zero $\phi^*_{i}\in L^1(\mathbb{R}^d)\cap L^{m_i}(\mathbb{R}^d)$ satisfying \eqref{lower bound for w*}, $i=1,2$.  For any $\alpha,\beta>0$ satisfying \eqref{definition of alpha}-\eqref{definition of alpha and beta}, it is possible to assume that $\alpha/\beta\geq 1$. If not, we can take $\alpha/\beta< 1$ if necessary. Then there holds
			\begin{equation*}{\label{}}
				\begin{split}
					1=&\varliminf_{j\rightarrow\infty}\left(\|\phi^*_{1j}\|^{\frac{\alpha}{\beta}}_1\|\phi^*_{2j}\|_1\right)\geq \left(\varliminf_{j\rightarrow\infty}\|\phi^*_{1j}\|^{\frac{\alpha}{\beta}}_1\right)\cdot \left(\varliminf_{j\rightarrow\infty}\|\phi^*_{2j}\|_1\right)\\
					\geq&\left(\varliminf_{j\rightarrow\infty}\|\phi^*_{1j}\|_1\right)^{\frac{\alpha}{\beta}}\cdot \|\phi^*_{2}\|_1\geq\|\phi^*_{1}\|^{\frac{\alpha}{\beta}}_1\|\phi^*_{2}\|_1,
				\end{split}
			\end{equation*}
and hence
			\begin{equation}{\label{inequality for psi in L1}}
					\|\phi^*_{1}\|^{\alpha}_1\|\phi^*_{2}\|^{\beta}_1\leq 1.
		\end{equation}
Since $\|\phi^*_{ij}\|_1$, $i=1,2$, is uniformly bounded for every $j\in\mathbb{N}^+$, we deduce from \eqref{inequality for phi} that
		\begin{equation*}	{\label{inequality for phi ij}}
				0\leq \phi^*_{ij}(R)\leq \mathcal{M}_i(R):=C_1\min\left\{R^{-d},R^{-d/{m_i}}\right\}\quad\text{for}\,\,R>0,\quad i=1,2.
		\end{equation*}	
So $\mathcal{M}_1(|x|)\in L^{p_1}(\mathbb{R}^d)$ with some $p_1\in\left(1/\left(1+2/d-1/m_2\right),m_1\right)$ and $\mathcal{M}_2(|x|)\in L^{p_2}(\mathbb{R}^d)$ with $ p_2=p_1/[(1+2/d)p_1-1]\in(1/\left(1+2/d-1/m_1\right),m_2)$. An application of the Hardy-Littlewood-Sobolev inequality results in
		\begin{equation*}{\label{}}
				\mathcal{H}[\mathcal{M}_1,\mathcal{M}_2]\leq\|\mathcal{M}_1\|_{p_1}\|\mathcal{M}_2\|_{p_2}<\infty.
		\end{equation*}
By utilizing the Lebesgue dominated convergence theorem and (\ref{weak convergence for f*}), we derive
		\begin{equation*}{\label{}}
				\lim_{j\rightarrow\infty}\mathcal{H}[\phi^*_{1j},\phi^*_{2j}]
				=\mathcal{H}[\phi^*_1,\phi^*_2],
		\end{equation*}
which implies that $\mathcal{H}[\phi^*_1,\phi^*_2]=\Lambda^*_{m_1,m_2}$. Moreover, $\mathcal{W}[\phi^*_1,\phi^*_2]\geq \Lambda^*_{m_1,m_2}$ is valid based on \eqref{lower bound for w*}-\eqref{inequality for psi in L1}. Consequently, it follows that $\mathcal{W}[\phi^*_1,\phi^*_2]=\Lambda^*_{m_1,m_2}$, where $(\phi^*_1,\phi^*_2)\in \Gamma_1\times\Gamma_2$ is a maximum for $\Lambda^*_{m_1,m_2}$ and    
		\begin{equation*}
			\|\phi^*_{1}\|^{\alpha}_1\|\phi^*_{2}\|^{\beta}_1= 1,\quad	 \|\phi^*_1\|_{m_1}=\|\phi^*_2\|_{m_2}=1.
		\end{equation*}
Hence, we have successfully completed the proof of this lemma.
	\end{proof}

\begin{remark}
Let $m_*=2d/(d+2)$.	For any $m_1=m_2=m\in(m_*,m^*]$, $h_1=h_2=h\in L^1(\mathbb{R}^d)\cap L^m(\mathbb{R}^d)$, $\alpha+\beta=[(1+2/d)m-2]/(m-1)\in(0,2/d]$ in \eqref{definiton of Lambda}, Lemma \ref{lemma maximizer for H 2} tells that  there exists a function $\phi^*\in L^1(\mathbb{R}^d)\cap L^{m}(\mathbb{R}^d)$ with $\|\phi^*\|_1=\|\phi^*\|_m=1$  satisfying the following variational problem
		\begin{equation*}{\label{}}
				C_c:=\sup_{0\neq h\in L^1(\mathbb{R}^d)\cap L^m(\mathbb{R}^d)}\frac{\mathcal{H}[h,h]}{\|h\|^{\alpha+\beta}_1\|h\|^{2-\alpha-\beta}_{m}}.
		\end{equation*}
Our result includes the case $m=m^*$ in \cite[Lemmas 3.2 and 3.3]{BCL09-CVPDE} or the case $m\in(m_*,m^*)$ in \cite[Theorem 2.1]{KNO2015-CVPDE}.
	\end{remark}
	
In order to create a pronounced curve for the situation when $m_1=m_2=m^*$, it is necessary to formulate an additional maximizing problem:
\begin{equation}{\label{definiton of Pi 2}}
		\Pi^*_{\theta_0}=\sup_{0\neq h_i\in \Gamma_i,i=1,2}\mathcal{V}[h_{1},h_{2}],
\end{equation}
where 
$$
\mathcal{V}[h_{1},h_{2}]:=\frac{\mathcal{H}[h_1,h_2]}{\|h_1\|_1\|h_2\|_1[\theta_0\| h_1\|^{-m^*}_1\|h_1\|^{m^*}_{m^*}+(1-\theta_0)\|h_2\|^{-m^*}_1\|h_2\|^{m^*}_{m^*}]}
$$
and 
$\theta_0=M^{m^*}_1/(M^{m^*}_1+M^{m^*}_2)\in(0,1)$ with $M_1,M_2>0$. 
Now we present a lemma to provide a simple inequality which will be used later.
	\begin{lemma}\label{lemma for Young inequality}	
		Let $a,b>0$, $\theta\in(0,1)$, and $\eta\in(0,1)$. Then	\begin{equation}{\label{inequality for a and b}}
			\begin{split}
				a^{\eta}b^{1-\eta}\leq A(\theta,\eta) \left[\theta a+(1-\theta) b\right]\quad\text{for}\quad a,b>0, \eta\in(0,1),
			\end{split}	
		\end{equation}
		where \begin{equation}{\label{definition of A}}
			\begin{split}
				A(\theta,\eta)=\left(\frac{\eta}{\theta}\right)^{\eta}\left(\frac{1-\eta}{1-\theta}\right)^{1-\eta}>0.
			\end{split}	
		\end{equation} The inequality in \eqref{inequality for a and b} is an equality if and only if $b=[\theta(1-\eta)]/[(1-\theta)\eta]a$. 
	\end{lemma}
	\begin{proof}
		The result is trivial by Young's inequality.
	\end{proof}

Given $\alpha,\beta>0$ in \eqref{definition of alpha}-\eqref{definition of beta} such that $\alpha+\beta=2/d$, and $\eta=(1-\alpha)/m^*\in(0,1)$. Since 
\begin{equation*}\label{inequality for H 2}
	\begin{split}
		\mathcal{H}[h_1,h_2]\leq &\Lambda^*_{m^*,m^*}\|h_1\|^{\alpha}_1\|h_2\|^{\beta}_1\|h_1\|^{1-\alpha}_{m^*}\|h_2\|^{1-\beta}_{m^*}\\
		=&\Lambda^*_{m^*,m^*}\|h_1\|_1\|h_2\|_1\left(\|h_1\|^{-m^*}_1\|h_1\|^{m^*}_{m^*}\right)^{(1-\alpha)/m^*}\left(\|h_2\|^{-m^*}_1\|h_2\|^{m^*}_{m^*}\right)^{(1-\beta)/m^*}\\
		\leq&A(\theta_0,\eta)\Lambda^*_{m^*,m^*}\|h_1\|_1\|h_2\|_1\left[\theta_0\|h_1\|^{-m^*}_1\|h_1\|^{m^*}_{m^*}+(1-\theta_0)\|h_2\|^{-m^*}_1\|h_2\|^{m^*}_{m^*}\right]
	\end{split}
\end{equation*}	
holds true by means of Lemmas \ref{lemma maximizer for H 2} and \ref{lemma for Young inequality}, we have $\Pi^*_{\theta_0}\leq A(\theta_0,\eta)\Lambda^*_{m^*,m^*}$. Hence \eqref{definiton of Pi 2} is well-defined. The following lemma shows the existence of a maximizer for $\Pi^*_{\theta_0}$. 
\begin{lemma}{\label{lemma maximizer for H 4}}
Let $\theta_0=M^{m^*}_1/(M^{m^*}_1+M^{m^*}_2)\in(0,1)$ with some $M_1,M_2>0$. There exists a pair of $(\psi^*_1,\psi^*_2)$ of non-negative, radially symmetric and non-increasing function such that it attains $\Pi^*_{\theta_0}$ with $\|\psi^*_{1}\|_1,\|\psi^*_{2}\|_1\leq 1$ and $\theta_0\|\psi^*_{1}\|^{m^*}_{m^*}+(1-\theta_0)\|\psi^*_{2}\|^{m^*}_{m^*}\leq1$. 
\end{lemma}
\begin{proof}
We let $(h_{ij})_{j\in\mathbb{N}}\in\Gamma_i$, $i=1,2$, be a  non-negative maximizing sequence such that	
\begin{equation*}	{\label{minimizing sequence for F 2}}
			\lim_{j\rightarrow\infty}\mathcal{V}[h_{1j},h_{2j}]=\Pi^*_{\theta_0}.
	\end{equation*}	
Denote $\psi_{ij}(x):=\lambda_{ij}h_{ij}(\mu_j x)$ with $\lambda_{ij}=\|h_{ij}\|^{-1}_1\mu^d_j$ and 
$$
\mu_{j}=\left[\theta_0\|h_{1j}\|^{-m^*}_1\|h_{1j}\|^{m^*}_{m^*}+(1-\theta_0)\|h_{2j}\|^{-m^*}_1\|h_{2j}\|^{m^*}_{m^*}\right]^{-{1}/(d-2)}.
$$ 
Then $\|\psi_{ij}\|_1=1$, $i=1,2$, $\theta_0\|\psi_{1j}\|^{-m^*}_1\|\psi_{1j}\|^{m^*}_{m^*}+(1-\theta_0)\|\psi_{2j}\|^{-m^*}_1\|\psi_{2j}\|^{m^*}_{m^*}=\theta_0\|\psi_{1j}\|^{m^*}_{m^*}+(1-\theta_0)\|\psi_{2j}\|^{m^*}_{m^*}=1$ and 
	\begin{equation*}{\label{}}
			\mathcal{V}[\psi_{1j},\psi_{2j}]=\mathcal{H}[h_{1j},h_{2j}].
	\end{equation*}
By rearranging of $\psi_{ij}$, we can assume that it is a non-negative, symmetric, and decreasing function with the following property
\begin{equation*}{\label{}}
			\lim_{j\rightarrow\infty}\mathcal{H}[\psi_{1j},\psi_{2j}]=\Pi^*_{\theta_0}.
	\end{equation*}
Then by the uniform bound of $\|\psi_{ij}\|_l$, $l\in\{1,{m^*}\}$, one asserts that $\psi_{ij}$ converges pointwise to some non-negative function $\psi^*_i\in L^1(\mathbb{R}^d)\cap L^{m^*}(\mathbb{R}^d)$ satisfying
	\begin{equation}{\label{lower bound for w* 2}}
		\begin{split}
		\|\psi^*_{1}\|_1,\|\psi^*_{2}\|_1\leq 1,\quad \theta_0\|\psi^*_{1}\|^{m^*}_{m^*}+(1-\theta_0)\|\psi^*_{2}\|^{m^*}_{m^*}\leq1.
		\end{split}
	\end{equation}	
Furthermore, 
$\lim_{j\rightarrow\infty}\mathcal{H}[\psi^*_{1j},\psi^*_{2j}]=\mathcal{H}[\psi^*_1,\psi^*_2]$ by the Lebesgue dominated convergence theorem. Hence 	$\mathcal{H}[\psi^*_1,\psi^*_2]=\Pi^*_{\theta_0}$ with $\psi^*_i\neq 0$, $i=1,2$, and we finish our proof.
\end{proof}
\begin{remark}
When $M_1=M_2$ in the sense that $\theta_{0}=1/2$, we claim that 
$$
\Pi^*_{\theta_0}=C_*,
$$
where $C_*$ is given by \eqref{definition of Cstar}. To see this, let us apply \cite[Theorem 9.8]{Lieb2001}, the Hardy-Littlewood-Sobolev inequality and Young's inequality to see that
\begin{equation}\label{inequality for mathcal H f1 f2}
\begin{split}
\vert \mathcal{H}[f_1,f_2]\vert\leq& \sqrt{\mathcal{H}[f_1,f_1]}\sqrt{\mathcal{H}[f_2,f_2]}\leq C_*\Vert f_1\Vert^{1/d}_1\Vert f_2\Vert^{1/d}_1\Vert f_1\Vert^{m^*/2}_{m^*}\Vert f_2\Vert^{m^*/2}_{m^*}\\
\leq&C_*\Vert f_1\Vert_1\Vert f_2\Vert_1\left[\frac{1}{2}\Vert f_1\Vert^{-m^*}_1\Vert f_1\Vert^{m^*}_{m^*}+\frac{1}{2}\Vert f_2\Vert^{-m^*}_1\Vert f_2\Vert^{m^*}_{m^*}\right],
\end{split}
\end{equation}
which implies that the inequality $\Pi^*_{\theta_0}\leq C_*$. This together with the fact that $C_*\leq \Pi^*_{\theta_0}$ yields $\Pi^*_{\theta_0}=C_*$. Furthermore, again by \cite[Theorem 9.8]{Lieb2001}, one can see that \eqref{inequality for mathcal H f1 f2} is equality if and only if $f_1=f_2$. As a consequence, we conclude from Lemma \ref{lemma maximizer for H 4} that  there exists a pair of nonnegative maximizing  functions $(\psi^*_1,\psi^*_2)$ with $\psi^*_1=\psi^*_2$ satisfying $\Vert \psi_i\Vert_1=\Vert \psi_i\Vert_{m^*}=1$, $i=1,2$, for $\Pi^*_{\theta_0}$ in the case $M_1=M_2$. This argument had been obtained in \cite[Lemma 3.3]{BCL09-CVPDE}.
\end{remark}
	
	\section{Sufficient condition on the global well-posedness of solutions}\label{Sufficient condition on well-posedness of solutions}
	In this section, a simple criterion on the global existence of free energy solution to \eqref{system} will be given. This criterion is weaker than that obtained in \cite[Lemma 2.3]{CK2021-ANA}. Inspired by \cite{BianLiu2013}, we first consider an  approximated system as
	\be \label{TSTC approxiamation}\begin{cases}
		\partial_t u_{i\epsilon} =  \Delta (u_{i\epsilon}+\epsilon)^{m_i}-\nabla \cdot ((u_{i\epsilon}+\epsilon)\nabla v_{j\epsilon}), &x\in \mathbb{R}^d, t>0, \\[0.2cm]
		-\Delta v_{j\epsilon} =J_{\epsilon}\ast u_{j\epsilon},  & x\in \mathbb{R}^d, t>0, \\[0.2cm]
		u_{i\epsilon}(x,0)=u^0_{i\epsilon}(x)>0, \,\,\,i,j=1,2,\,\,i\neq j,\,\,& x\in \mathbb{R}^d.
	\end{cases}
	\ee
	Here
	$J_{\epsilon}(x)=J\left(x/\epsilon\right)/\epsilon^d$, $ J(x)=c_dd(d-2)(|x|^2+1)^{-(d+2)/d}$
	and $\|J_{\epsilon}\|_1=1$. Then $v_{j\epsilon}$ can be expressed by
	$$
	v_{j\epsilon}(x,t)=\mathcal{K}_{\epsilon}\ast u_{j\epsilon},\quad j=1,2,
	$$
	where $\mathcal{K}_{\epsilon}(|x|)=c_d(|x|^2+\epsilon^2)^{-(d-2)/2}$. Noted that
	$$
	\nabla \mathcal{K}_{\epsilon}=-c_d(d-2)\frac{x}{(|x|^2+\epsilon^2)^{\frac{d}{2}}},\quad\,\Delta \mathcal{K}_{\epsilon}=-J_{\epsilon}.
	$$
	
	Let $u^{0}_{i\epsilon}\in C^{\infty}(\mathbb{R}^d)$ be an approximation of $u^0_{i}$ with a sequence of mollifiers. Then there exists a constant $\epsilon_0>0$ such that for any $\epsilon\in(0,\epsilon_0)$, the constructed $u^{0}_{i\epsilon}$ satisfies
	$$\|u^{0}_{i\epsilon}\|_1=\|u^0_{i}\|_1=M_i,$$
	$$u^{0}_{i\epsilon}\in L^{r}(\mathbb{R}^d)\quad \text{and}\quad \|u^{0}_{i\epsilon}-u^0_i\|_{r}\rightarrow 0\,\,\,\text{for all}\,\,\,r\geq 1,i=1,2.$$
	
According to the local existence of strong solution to one-species chemotaxis system  {\cite[Proposition 4.1]{S2006-DIE}}, it follows that
	\begin{lemma}\label{Local existence for AS}Let $m_1,m_2>1$. Then there exists $T^{\epsilon}_{\max}\in(0,\infty]$, denoting the maximal existence time, such that the system (\ref{TSTC approxiamation}) has a unique non-negative strong solution $u_{i\epsilon}\in \mathbb{W}^{2,1}_{r}(Q_T) $, $i=1,2$, for some $r>1$, where $Q_T=\mathbb{R}^d\times (0,T)$, $T\in(0,T^{\epsilon}_{\max}]$, and
		$$\mathbb{W}^{2,1}_{r}(Q_T):=\{u\in L^r(0,T;W^{2,r}(\mathbb{R}^d))\cap W^{1,r}(0,T;L^{r}(\mathbb{R}^d))\}.$$
		Moreover, if  $T^{\epsilon}_{\max}<\infty$, then
		$$
		\lim_{t\rightarrow T^{\epsilon}_{\max}}\sum\limits^2_{i=1}\|u_{i\epsilon}(\cdot,t)\|_{\infty}=\infty.
		$$
	\end{lemma}
	
The next lemma is looking for sufficient conditions, in terms of $L^{l_i}$-bound on $u_{i\epsilon}$ with some large $l_i>1$, guaranteeing global existence.
	\begin{lemma}\label{lemma Lp bound for u and Lq for wq}
		Let $m_1,m_2>1$ satisfy $1/m_1+1/m_2<1+2/d$, and let $T>0$. Assume that there exists a positive constant $C_0>0$ such that
		\begin{equation*}\label{}
			\begin{split}
				\|u_{i\epsilon}(\cdot,t)\|_{k_i}&<C_0,\quad i=1,2,
				\quad{for\,\,all}\,\,t>0,
			\end{split}
		\end{equation*}
		where $k_i=m_j/[(1+2/d)m_j-1]$, $i\neq j\in\{1,2\}$. Then for some large $l_i>1$, $i=1,2$, one obtains the existence of a constant $C>0$ such that 	 
		\begin{equation}\label{Linfty estimate for u}
			\|u_{i\epsilon}\|_{L^{\infty}\left(0,T;L^1_+(\mathbb{R}^d)\cap L^{l_i}(\mathbb{R}^d)\right)}\leq C\quad\text{for}\,\,t\in(0,T)
		\end{equation}
		and
		\begin{equation}\label{L2 estimate for w4}
			\left\|\nabla u^{\frac{l_i+m_i-1}{2}}_{i\epsilon}\right\|_{L^2(0,T;L^2(\mathbb{R}^d))}\leq C\quad\text{for}\,\,t\in(0,T).
		\end{equation}
	\end{lemma}
	\begin{proof}
		Multiplying $(\ref{TSTC approxiamation})_i$ by $l_iu^{l_i-1}_{i\epsilon}$, and summing them up,  one can see that
		\begin{equation}\label{inequality for uk and wl}
			\begin{split}
				\sum\limits^2_{i=1}\frac{d}{dt}&\int_{\mathbb{R}^d}u^{l_i}_{i\epsilon}dx+\sum\limits^2_{i=1}\frac{4m_il_i(l_i-1)}{(l_i+m_i-1)^2}\int_{\mathbb{R}^d}\left|\nabla u^{\frac{l_i+m_i-1}{2}}_{i\epsilon}\right|^2dx\\[1mm]
				\leq&-\sum\limits^2_{i=1}(l_i-1)\int_{\mathbb{R}^d}u^{l_i}_{i\epsilon}\Delta v_{j\epsilon}dx-\epsilon\sum\limits^2_{i=1} l_i \int_{\mathbb{R}^d}u^{l_i-1}_{i\epsilon}\Delta v_{j\epsilon}dx\\[1mm]
				=&\sum\limits^2_{i=1}(l_i-1)\int_{\mathbb{R}^d}u^{l_i}_{i\epsilon}J_{\epsilon}\ast u_{j\epsilon}dx+\epsilon\sum\limits^2_{i=1} l_i\int_{\mathbb{R}^d}u^{l_i-1}_{i\epsilon}J_{\epsilon}\ast u_{j\epsilon}dx\\[1mm]
				=&:I_1+I_2,\quad i\neq j\in\{1,2\}.
			\end{split}
		\end{equation}
		Applying H\"{o}lder's inequality with  $r_i>1$,  $r'_i=r_i/(r_i-1)$, $i=1,2$, it yields
		\begin{equation}\label{right side of uk and wl1}
			\begin{split}
				I_1=&\sum\limits^2_{i=1}(l_i-1)\int_{\mathbb{R}^d}u^{l_i}_{i\epsilon}J_{\epsilon}\ast u_{j\epsilon}dx\leq	 \sum\limits^2_{i=1}(l_i-1)\|u_{i\epsilon}\|^{l_i}_{l_ir_i}\|J_{\epsilon}\ast u_{j\epsilon}\|_{r'_i}\\
				\leq&\sum\limits^2_{i=1}(l_i-1)\|u_{i\epsilon}\|^{l_i}_{l_ir_i}\|u_{j\epsilon}\|_{r'_i},\quad i\neq j\in\{1,2\}.
			\end{split}
		\end{equation}
		
		Now we would like to use  the left-hand side of \eqref{inequality for uk and wl} to dominate the terms at the right-hand side of \eqref{right side of uk and wl1}. For this purpose, the values of $r_i$ should be chosen carefully. We claim that there exist $\bar{l}_i>1$, $r_i>1$, such that for any $l_i>\bar{l}_i$, $i=1,2$, and $j\neq i\in\{1,2\}$, the following inequalities are established:
		\be{\label{choice of k and l 111}}
		l_i>\max\left\{k_i+1,\frac{d}{2}-m_i\right\},
		\ee
		
		\be{\label{choice of k and l 222}}
		\begin{split}
			\frac{1}{r_i}<&1-\frac{d-2}{(l_j+m_j-1)d},
		\end{split}
		\ee
		
		\be \label{choice of k and l 333}
		\frac{1}{r_i}>\max\left\{1-\frac{1}{k_j},
		\frac{d-2}{d}\cdot\frac{l_i}{l_i+m_i-1}\right\},
		\ee
		
		\be{\label{choice of k and l 444}}
		\begin{split}	
			\frac{\frac{l_i}{k_i}-\frac{1}{r_i}}{1-\frac{d}{2}+\frac{(l_i+m_i-1)d}{2k_i}}+	 \frac{\frac{1}{k_j}-1+\frac{1}{r_i}}{1-\frac{d}{2}+\frac{(l_j+m_j-1)d}{2k_j}}<\frac{2}{d}.
		\end{split}
		\ee
		
		We give a short explanation for above claim. Choose $r_i>1$ satisfying
		$r_i<d/(d-2)$, $i=1,2$,	and set
		\begin{equation}{\label{relationship betweeen k and l}}
			\begin{split}	
				l_j:=\frac{(l_i+m_i-1)k_j}{k_i}-m_j+1,\quad i\neq j\in\{1,2\}.
			\end{split}
		\end{equation}
		Then as a result of  \eqref{relationship betweeen k and l}, the denominators at the left-hand sides of	\eqref{choice of k and l 444} are the same, which actually implies that 	\eqref{choice of k and l 444} comes true by the following fact 
		\begin{equation*}{\label{relationship betweeen p and q1}}
			\begin{split}
				\frac{1}{k_j}<\frac{2}{d}+\frac{m_i-1}{k_i},\quad i\neq j\in\{1,2\}.
			\end{split}
		\end{equation*}
		Moreover, the assertions of  (\ref{choice of k and l 111})-\eqref{choice of k and l 222} easily hold out if one takes  $l_i\geq \bar{l}_i$ sufficiently large with some $\bar{l}_i>1$, $i=1,2$.
		
		It should be pointed out that if $r_1$ satisfies both (\ref{choice of k and l 222}) and (\ref{choice of k and l 333}), then it is necessary to have
		\begin{equation*}\label{inequality for kl1}
			1-\frac{d-2}{(l_j+m_j-1)d}>1-\frac{1}{k_j},\quad i\neq j\in\{1,2\},
		\end{equation*}
		and
		\begin{equation*}\label{inequality for kl2}
			1-\frac{d-2}{(l_j+m_j-1)d}>\frac{d-2}{d}\cdot\frac{l_i}{l_i+m_i-1},\quad i\neq j\in\{1,2\}.
		\end{equation*}		
		However,  the above two inequalities are trivial from
		\begin{equation*}
			\begin{split}
				1-\frac{d-2}{(l_j+m_j-1)d}>\frac{d-2}{d}>1-\frac{1}{k_j},\quad j\in\{1,2\}.
			\end{split}
		\end{equation*}		
		Therefore, the choice for $r_i$, $i=1,2$, is available, and our claim has been proved.
		
		By the choices of $l_i$ and $r_i$, $i=1,2$, in \eqref{choice of k and l 111}-\eqref{choice of k and l 444}, we would like to estimate terms at the right-hand side of \eqref{right side of uk and wl1} now. The choices of (\ref{choice of k and l 111})-(\ref{choice of k and l 333}) ensure
		\begin{equation*}
			\begin{split}
				k_i<l_ir_i<\frac{(l_i+m_i-1)d}{d-2},\quad i\in\{1,2\},
			\end{split}
		\end{equation*}
		and
		\begin{equation*}
			\begin{split}
				k_j<r'_i<\frac{(l_j+m_j-1)d}{d-2},\quad i\neq j\in\{1,2\},
			\end{split}
		\end{equation*}
		which together with the Gagliardo-Nirenberg inequality (\cite{BianLiu2013,L2023-CMS}), the uniform bound assumptions on $\|u_{i\epsilon}\|_{k_i}$, $i=1,2$, and  \eqref{choice of k and l 444} implies that there exist constants $c_1>0$ and $c_2>0$ such that for any $\eta>0$,
		\begin{equation}\label{inequality for I1}
			\begin{split}
				I_1\leq &\sum\limits^2_{i=1}(l_i-1)\big\|u_{i\epsilon}\big\|^{l_i}_{l_ir_i}\big\| u_{j\epsilon}\big\|_{r'_i} \\	 \leq&c_1\sum\limits^2_{i=1}\big\|u_{i\epsilon}\big\|^{l_i(1-\theta_i)}_{k_i}\left\|\nabla u^{\frac{l_i+m_i-1}{2}}_{i\epsilon}\right\|^{l_i\frac{2\theta_i}{l_i+m_i-1}}_{2}\cdot \big\|u_{j\epsilon}\big\|^{1-\theta_j}_{k_j}\left\|\nabla u^{\frac{l_j+m_j-1}{2}}_{j\epsilon}\right\|^{\frac{2\theta_j}{l_j+m_j-1}}_{2}\\[1mm]
				\leq&c_1\sum\limits^2_{i=1}C^{l_i(1-\theta_i)+1-\theta_j}_0\left\|\nabla u^{\frac{l_i+m_i-1}{2}}_{i\epsilon}\right\|^{l_i\frac{2\theta_i}{l_i+m_i-1}}_{2}\left\|\nabla u^{\frac{l_j+m_j-1}{2}}_{j\epsilon}\right\|^{\frac{2\theta_j}{l_j+m_j-1}}_{2}\\[1mm]
				\leq&c_1\sum\limits^2_{i=1}C^{l_i(1-\theta_i)+1-\theta_j}_0\left\|\nabla u^{\frac{l_i+m_i-1}{2}}_{i\epsilon}\right\|^{\frac{\frac{l_i}{k_i}-\frac{1}{r_i}}{\frac{1}{d}-\frac{1}{2}+\frac{l_i+m_i-1}{2m_i}}}_{2}\left\|\nabla u^{\frac{l_j+m_j-1}{2}}_{j\epsilon}\right\|^{ \frac{\frac{1}{k_j}-1+\frac{1}{r_i}}{\frac{1}{d}-\frac{1}{2}+\frac{l_j+m_j-1}{2k_j}}}_{2}\\
				\leq &\eta \sum\limits^2_{i=1}\left\|\nabla u^{\frac{l_i+m_i-1}{2}}_{i\epsilon}\right\|^{2}_{2}+c_2, \quad i\neq j\in\{1,2\},
			\end{split}
		\end{equation}
		where
		\begin{equation*}\label{definition of theta1}
			\begin{split}
				\theta_i=&\frac{l_i+m_i-1}{2}\frac{\frac{1}{k_i}-\frac{1}{l_ir_i}}{\frac{1}{d}-\frac{1}{2}+\frac{l_i+m_i-1}{2k_i}}\in(0,1),\quad i\in\{1,2\},\\[1mm]
				\theta_j=&\frac{l_j+m_j-1}{2}\frac{\frac{1}{k_j}-\frac{1}{r'_i}}{\frac{1}{d}-\frac{1}{2}+\frac{l_j+m_j-1}{2k_j}}\in(0,1),\quad j\neq i\in\{1,2\}.
			\end{split}
		\end{equation*}

		Following similar procedure above to control $I_2$ in \eqref{inequality for uk and wl}, there exists a constant $c_3>0$ such that
		\begin{equation}\label{inequality for I2}
			\begin{split}
				I_2&\leq\eta \sum\limits^2_{i=1}\left\|\nabla u^{\frac{l_i+m_i-1}{2}}_{i\epsilon}\right\|^{2}_{2}+c_3, \quad i\neq j\in\{1,2\}.
			\end{split}
		\end{equation}

		Then by means of \eqref{inequality for I1}, \eqref{inequality for I2} and  \eqref{inequality for uk and wl}, it indicates that for some small $\eta>0$,
		\begin{equation*}\label{inequality for uk and wl smaller than 0}
			\begin{split}
				\sum\limits^2_{i=1}\frac{d}{dt}&\int_{\mathbb{R}^d}u^{l_i}_{i\epsilon}dx+\sum\limits^2_{i=1}\frac{2m_il_i(l_i-1)}{(l_i+m_i-1)^2}\int_{\mathbb{R}^d}\left|\nabla u^{\frac{l_i+m_i-1}{2}}_{i\epsilon}\right|^2dx\leq c_2+c_3.
			\end{split}
		\end{equation*}
		Integrating the above inequality over $s\in(0,t)$, it leads to the following inequality
		\begin{equation*}\label{inequality with time for uk and wl}
			\begin{split}
				\sum\limits^2_{i=1}\int_{\mathbb{R}^d}u^{l_i}_{i\epsilon}dx+&\sum\limits^2_{i=1}\frac{2m_il_i(l_i-1)}{(l_i+m_i-1)^2}\int^t_0\int_{\mathbb{R}^d}\left|\nabla u^{\frac{l_i+m_i-1}{2}}_{i\epsilon}\right|^2dxds\\
				\leq&\sum\limits^2_{i=1} \int_{\mathbb{R}^d}(u^{0}_{i\epsilon})^ldx+(c_2+c_3)T,\quad t\in(0,T).
			\end{split}
		\end{equation*}
		Hence this lemma has been completed by 	the fact that $(u^0_\epsilon, w^0_\epsilon)\in (L^{1}({\mathbb{R}^d})\cap L^{\infty}(\mathbb{R}^d))^2$.
	\end{proof}
	The following lemma gives a uniform bound for the solutions of system \eqref{TSTC approxiamation}.
	\begin{lemma}\label{lemma infity bound for u and w}
		Let $T>0$, and let the same assumption hold in Lemma \ref{lemma Lp bound for u and Lq for wq}. Then there exists a positive constant $C>0$ independent of $\epsilon$ such that
		\begin{equation}\label{L infty for u and w}
			\begin{split}
				\sup_{0<t<T}&\sum\limits^2_{i=1}\|u_{i\epsilon}(t)\|_{r}\leq C,\,\,\,r\in[1,\infty],
			\end{split}
		\end{equation}
		\begin{equation}\label{L infty for nabla v and nabla w}
			\begin{split}
				\sup_{0<t<T}&\sum\limits^2_{j=1}\|\nabla v_{j\epsilon}(t)\|_{r}\leq C,\,\,\,r\in(d/(d-2),\infty].
			\end{split}
		\end{equation}
		Moreover, \eqref{system} possesses a global free energy solution.
	\end{lemma}
	\begin{proof}
		It follows from Lemma \ref{lemma Lp bound for u and Lq for wq} that \eqref{Linfty estimate for u}-\eqref{L2 estimate for w4} hold out, and then applying  a bootstrap iterative technique in \cite[Lemma 3.4]{L2023-CMS}, the $L^{\infty}$-bound for the solution could be obtained. Consequently, combining $L^{\infty}$-bound with $L^1$-norm results in \eqref{L infty for u and w}. While \eqref{L infty for nabla v and nabla w} follows from fundamental solution for $v_j$, $j=1,2$, \eqref{L infty for u and w} and the weak Young inequality \cite[formula (9), pp. 107]{Lieb2001}.
		
		We multiply $(\ref{TSTC approxiamation})_i$ by $\partial_t(u_{i\epsilon}+\epsilon)^{m_i}$, and then integrate the identity over $x\in\mathbb{R}^d$ to arrive at
		\begin{equation*}\label{time ienquality for nabla u2}
			\begin{split}
				\frac{1}{2}\sum\limits^2_{i=1}&\frac{d}{dt}\int_{\mathbb{R}^d}\left|\nabla(u_{i\epsilon}+\epsilon)^{m_i}\right|^2dx+\sum\limits^2_{i=1}
				\frac{4m_i}{(m_i+1)^2}\int_{\mathbb{R}^d}\left|\partial_t(u_{i\epsilon}+\epsilon)^{\frac{m_i+1}{2}}\right|^2dx\\
				=&-\sum\limits^2_{i=1}\frac{2m_i}{m_i+1}\int_{\mathbb{R}^d}(u_{i\epsilon}+\epsilon)^{\frac{m_i-1}{2}}\nabla \cdot ((u_{i\epsilon}+\epsilon)\nabla v_{j\epsilon})\cdot\partial_t(u_{i\epsilon}+\epsilon)^{\frac{m_i+1}{2}}dx\\
				\leq&\sum\limits^2_{i=1} \frac{2m_i}{(m_i+1)^2}\int_{\mathbb{R}^d}\left|\partial_t(u_{i\epsilon}+\epsilon)^{\frac{m_i+1}{2}}\right|^2dx+c_1\sum\limits^2_{i=1}\int_{\mathbb{R}^d}\left|\nabla(u_{i\epsilon}+\epsilon)^{\frac{m_i+1}{2}}\right|^2dx\\
				&+c_1\sum\limits^2_{j=1}\int_{\mathbb{R}^d}\left|\Delta v_{j\epsilon}\right|^2dx,\quad i\neq j\in\{1,2\},
			\end{split}
		\end{equation*}
		where we have used Young's inequality, \eqref{L infty for u and w}-\eqref{L infty for nabla v and nabla w} and $c_1>0$.
		Integrating above inequality over $(0,t)$, $t\in(0,T)$, it leads to
		\begin{equation}\label{time ienquality for nabla w2}
			\begin{split}
				\frac{1}{2}&\sum\limits^2_{i=1}\big\|\nabla (u_{i\epsilon}+\epsilon)^{m_i}\big\|^2_2+	 \sum\limits^2_{i=1}\frac{2m_i}{(m_i+1)^2}\left\|\partial_t(u_{i\epsilon}+\epsilon)^{\frac{m_i+1}{2}}\right\|^2_{L^2(0,t;L^2(\mathbb{R}^d))}\\[1mm]
				\leq&\sum\limits^2_{i=1}\left\|\nabla (u^0_{i\epsilon}+\epsilon)^{m_i}\right\|^2_2+c_1\sum\limits^2_{i=1}\left\|\nabla(u_{i\epsilon}+\epsilon)^{\frac{m_i+1}{2}}\right\|^2_{L^2(0,t;L^2(\mathbb{R}^d))}\\
				&+c_1\sum\limits^2_{j=1}
				\left\|\Delta v_{j\epsilon}\right\|^2_{L^2(0,t;L^2(\mathbb{R}^d))},\quad i\neq j\in\{1,2\}.
			\end{split}
		\end{equation}		
		Multiplying $(\ref{TSTC approxiamation})_i$ by $u_{i\epsilon}$, it yields
		\begin{equation}\label{}
			\begin{split}
				\frac{1}{2}\frac{d}{dt}\int_{\mathbb{R}^d}u^2_{i\epsilon}+\frac{4m_i}{(m_i+1)^2}\int_{\mathbb{R}^d}|\nabla (u_{i\epsilon}+\epsilon)^{\frac{m_i+1}{2}}|^2dx=-\int_{\mathbb{R}^d}\left(\frac{1}{2}u_{i\epsilon}+\epsilon \right)u_{i\epsilon}\Delta v_{j\epsilon}dx,
			\end{split}
		\end{equation}	
		and then it follows from \eqref{L infty for u and w} that
		\begin{equation}\label{time ienquality for nabla w3}
			\begin{split}
				\left\|\nabla(u_{i\epsilon}+\epsilon)^{\frac{m_i+1}{2}}\right\|^2_{L^2(0,T;L^2(\mathbb{R}^d))}\leq c_2
			\end{split}
		\end{equation}	
		by integrating over time with some $c_2>0$. By means of \eqref{time ienquality for nabla w2}, \eqref{time ienquality for nabla w3} and the assumptions on the initial data, there is a constant $c_3>0$ such that
		\begin{equation*}\label{patial u L1}
			\begin{split}
				&\sum\limits^2_{i=1}\left\|\partial_tu^{m_i}_{i\epsilon}\right\|^2_{L^2(0,T;L^2(\mathbb{R}^d))}+\sum\limits^2_{i=1}\sup_{0<t<T}\left\|\nabla u^{m_i}_{i\epsilon}\right\|^2_2\\
				&\leq \sum\limits^2_{i=1} \left\|\partial_t(u_{i\epsilon}+\epsilon)^{m_i}\right\|^2_{L^2(0,T;L^2(\mathbb{R}^d))}+\sum\limits^2_{i=1}\sup_{0<t<T}\left\|\nabla (u_{i\epsilon}+\epsilon)^{m_i}\right\|^2_2\leq c_3.
			\end{split}
		\end{equation*}	
		
		Invoking the Aubin-Lions lemma and \eqref{L infty for u and w}, one can find subsequences of $\{u_{i\epsilon_n}\}$, $i=1,2$, such that
		\begin{equation*}\label{convergence almost everywhere for un and wn}
			u_{i\epsilon_{n}}\rightarrow u_i\,\text{strongly in}\,\,C(0,T;L^{r}_{loc}(\mathbb{R}^d)),\,\,\forall r\in[1,\infty),\quad i=1,2,\\[1mm]
		\end{equation*}
		\begin{equation*}\label{}
			u_{i\epsilon_{n}}\rightarrow u_i\,\,\text{a.e. in}\,\,\mathbb{R}^d\times(0,T),\quad i=1,2,\\[1mm]
		\end{equation*}
		\begin{equation*}
			\nabla (u_{i\epsilon_{n}}+\epsilon_{n})^{\frac{m_i+1}{2}}\rightharpoonup\,\,\nabla u^{\frac{m_i+1}{2}}_i\,\,\,\text{weakly}\,\,{in}\,\,L^2(0,T; L^{2}(\mathbb{R}^d)),\quad i=1,2.\\[1mm]
		\end{equation*}
		In view of \eqref{L infty for nabla v and nabla w}, we can also extract subsequences $\{v_{j\epsilon_n}\}$ such that
		\begin{equation*}
			\nabla v_{j\epsilon_{n}}(t)\rightarrow\nabla v_j(t)\,\,\,\text{weakly-}\ast\,\,{in}\,\,L^\infty(0,T; L^{r}(\mathbb{R}^d)),\,\,\,\forall r\in(d/(d-1),\infty],\quad j=1,2.\\[1mm]
		\end{equation*}
		The details of the proof can be found in \cite[Lemma 2.4]{CK2021-ANA}. Hence, a global weak solution of Cauchy problem \eqref{system} was obtained and satisfies Definition \ref{weak solution} $(i)-(ii)$. The free energy inequality \eqref{free energy inequality} has been proved in \cite[Lemma 2.5]{CK2021-ANA} if  $u^{(2m_i-1)/2}_i\in L^2(0,T;H^1(\mathbb{R}^d))$, $i=1,2$.  However, based on \eqref{L infty for u and w}, this regularity can be obtained by integrating \eqref{inequality for uk and wl} with $l_i=m_i$, $i=1,2$, over time.
	\end{proof}
	
	\section{Global existence}\label{global existence}
	
	In this section, we consider the global existence of solutions in two different cases: $(\romannumeral1)$\,\,$p\geq1$, $q\geq1$ and $r>1$ with $m_1,m_2\neq m^*$; $(\romannumeral2)$\,\,$m_1=m_2=m^*$. Chosen $\eta={\frac{1-\alpha}{m_1}/\left(\frac{1-\alpha}{m_1}+\frac{1-\beta}{m_2}\right)}\in(0,1)$ with $\alpha,\beta> 0$ satisfying \eqref{definition of alpha}-\eqref{definition of alpha and beta}, and denote $$
	\Lambda^*_{m_1,m_2,\theta}=A^{\frac{1-\alpha}{m_1}+\frac{1-\beta}{m_2}}\left(\theta,\frac{\frac{1-\alpha}{m_1}}{\frac{1-\alpha}{m_1}+\frac{1-\beta}{m_2}}\right)\Lambda^*_{m_1,m_2}=A^{\frac{1-\alpha}{m_1}+\frac{1-\beta}{m_2}}\left(\theta,\eta\right)\Lambda^*_{m_1,m_2},
	$$
	where $\Lambda^*_{m_1,m_2}$ is given by \eqref{definiton of Lambda}. In the case of $(\romannumeral1)$, it should be noted that $(1-\alpha)/m_1+(1-\beta)/m_2>1$. Define an auxiliary function
	$f(x)=x-\Lambda^*_{m_1,m_2,\theta}x^{\frac{1-\alpha}{m_1}+\frac{1-\beta}{m_2}}$ for $x>0$. It has been observed that the function $f(x)$ reaches its maximum value at the point $x_0$, which is given by \eqref{definition of x0}.
	
	In the next two lemmas, our objective is to prove the uniform bound for $\|u_i(\cdot,t)\|_{m_i}$, $i=1,2$, for all $t>0$. This will  imply that solution exists globally due to the criterion in Section \ref{Sufficient condition on well-posedness of solutions}.
	\begin{lemma}\label{lemma for uniform for u in Lm 1}	
		Given  $T>0$, $\theta\in(0,1)$. Let $p\geq1$, $q\geq1$ and $r>1$. Assume that initial data ${\textbf{u}^0}$ satisfies \eqref{total mass of solution}, \eqref{assumption on initial data} and
		\begin{equation}\label{assumption on F 1}
				\mathcal{F}[u^0_1,u^0_2]< \frac{c_d}{\kappa^{1+\gamma_1}_1\kappa^{1+\gamma_2}_2M^{\gamma_1}_1M^{\gamma_2}_2}f(x_0)
		\end{equation}	
		as well as
		\begin{equation}\label{assumption on ui in Lmi 1}
				\theta\kappa^{m_1}_1\|u^0_1\|^{m_1}_{m_1}+(1-\theta)\kappa^{m_2}_2\|u^0_2\|^{m_2}_{m_2}<\frac{x_0}{\kappa^{\gamma_1}_1\kappa^{\gamma_2}_2M^{\gamma_1}_1M^{\gamma_2}_2},
		\end{equation}	
		where
		\begin{equation}\label{definition of kappa}
			\begin{split}
				\kappa_1=&c_d^{\frac{m_2}{m_1+m_2-m_1m_2}}\left[(m_1-1)\theta\right]^{\frac{m_2-1}{m_1+m_2-m_1m_2}}\left[(m_2-1)(1-\theta)\right]^{\frac{1}{m_1+m_2-m_1m_2}},\\
				\kappa_2=&c_d^{\frac{m_1}{m_1+m_2-m_1m_2}}\left[(m_1-1)\theta\right]^{\frac{1}{m_1+m_2-m_1m_2}}\left[(m_2-1)(1-\theta)\right]^{\frac{m_1-1}{m_1+m_2-m_1m_2}},
			\end{split}
		\end{equation}
		and
		\begin{equation*}\label{}
				\gamma_1=\frac{\alpha}{	\frac{1-\alpha}{m_1}+\frac{1-\beta}{m_2}-1}> 0,\quad 	\gamma_2=\frac{\beta}{	 \frac{1-\alpha}{m_1}+\frac{1-\beta}{m_2}-1}>0.
		\end{equation*}	
		Then there exists a constant $c>0$ such that free energy solution ${\textbf{u}}$ of \eqref{system} with the initial data ${\textbf{u}^0}$ fulfills
		\begin{equation*}
				\|u_i(\cdot,t)\|_{m_i}\leq c\quad\text{for all}\,\,i=1,2,\,\,\text{and}\,\, t\in(0,T).
		\end{equation*}
	\end{lemma}
	\begin{proof}
		Using the definition of \eqref{definition of kappa}, we can derive that
		\begin{equation*}\label{}
				\theta=\frac{\kappa_1\kappa_2}{c_d(m_1-1)\kappa^{m_1}_1},\quad 	1-\theta=\frac{\kappa_1\kappa_2}{c_d(m_2-1)\kappa^{m_2}_2}.
		\end{equation*}	
	By transforming the local free energy solution $u_i$ into $h_i=\kappa_i u_i$, we discover that $h_i\in\Gamma_i$ with $\|h_i\|_1=\kappa_i\|u^0_i\|_1=\kappa_iM_i$, $i=1 ,2$, and that
		\begin{equation*}\label{}
				\mathcal{F}[u_1,u_2]=\sum^{2}_{i=1}\frac{1}{m_i-1}\|u_i\|^{m_i}_{m_i}-c_d\mathcal{H}[u_1,u_2]=\frac{c_d}{\kappa_1\kappa_2}\mathcal{G}[h_1,h_2],
		\end{equation*}	
where the function $\mathcal{G}$ is defined as
		\begin{equation*}\label{definition of G}
				\mathcal{G}[h_1,h_2]:=	\theta\|h_1\|^{m_1}_{m_1}+(1-\theta)\|h_2\|^{m_2}_{m_2}-\mathcal{H}[h_1,h_2].
		\end{equation*}	
According to Lemma \ref{lemma for Young inequality}, the definition of $\Lambda^*_{m_1,m_2}$ implies that
		\begin{equation}\label{inequality for H 1}
			\begin{split}
				\left|\mathcal{H}[h_1,h_2]\right|\leq& \Lambda^*_{m_1,m_2} \|h_1\|^{\alpha}_1\|h_2\|^{\beta}_1\|h_1\|^{1-\alpha}_{m_1}\|h_2\|^{1-\beta}_{m_2}\\
				=&\Lambda^*_{m_1,m_2}\|h_1\|^{\alpha}_{1}\|h_2\|^{\beta}_{1}\left(\|h_1\|^{m_1}_{m_1}\right)^{\frac{1-\alpha}{m_1}}\left(\|h_2\|^{m_2}_{m_2}\right)^{\frac{1-\beta}{m_2}}\\
				\leq&\Lambda^*_{m_1,m_2}A^{\frac{1-\alpha}{m_1}+\frac{1-\beta}{m_2}}\left(\theta,\eta\right)\|h_1\|^{\alpha}_{1}\|h_2\|^{\beta}_{1}\\
    &\cdot\left(\theta\|h_1\|^{m_1}_{m_1}+(1-\theta)\|h_2\|^{m_2}_{m_2}\right)^{\frac{1-\alpha}{m_1}+\frac{1-\beta}{m_2}}\\
				=&\Lambda^*_{m_1,m_2,\theta}\|h_1\|^{\alpha}_{1}\|h_2\|^{\beta}_{1}B^{\frac{1-\alpha}{m_1}+\frac{1-\beta}{m_2}},
			\end{split}
		\end{equation}
with $B:=\theta\|h_1\|^{m_1}_{m_1}+(1-\theta)\|h_2\|^{m_2}_{m_2}$. This leads to
		\begin{equation*}\label{inequality for G 1}
			\begin{split}
				\|h_1\|^{\gamma_1}_1\|h_2\|^{\gamma_2}_1\mathcal{G}[h_1,h_2]\geq& \|h_1\|^{\gamma_1}_1\|h_2\|^{\gamma_2}_1B-\Lambda^*_{m_1,m_2,\theta}\|h_1\|^{\gamma_1+\alpha}_1\|h_2\|^{\gamma_2+\beta}_1B^{\frac{1-\alpha}{m_1}+\frac{1-\beta}{m_2}}\\
				=&\|h_1\|^{\gamma_1}_1\|h_2\|^{\gamma_2}_1B-\Lambda^*_{m_1,m_2,\theta}\left(\|h_1\|^{\gamma_1}_1\|h_2\|^{\gamma_2}_1B\right)^{\frac{1-\alpha}{m_1}+\frac{1-\beta}{m_2}}\\
				=&f(\|h_1\|^{\gamma_1}_1\|h_2\|^{\gamma_2}_1B).
			\end{split}
		\end{equation*}
For a small positive value of $\delta>0$, we obtain from \eqref{assumption on F 1} and the non-increasing property of $\mathcal{F}$ that
		\begin{equation}\label{inequality for f}
			\begin{split}
				(1-\delta)f(x_0)				\geq&\frac{\kappa_1\kappa_2\|h_1\|^{\gamma_1}_1\|h_2\|^{\gamma_2}_1\mathcal{F}[u^0_1,u^0_2]}{c_d}\\
				\geq&\frac{\kappa_1\kappa_2\|h_1\|^{\gamma_1}_1\|h_2\|^{\gamma_2}_1\mathcal{F}[u_1(\cdot,t),u_2(\cdot,t)]}{c_d}\\
				=&\|h_1\|^{\gamma_1}_1\|h_2\|^{\gamma_2}_1\mathcal{G}[h_1,h_2]\\
				\geq&f(\|h_1\|^{\gamma_1}_1\|h_2\|^{\gamma_2}_1B)\quad\text{for}\,\,t\in(0,T).
			\end{split}
		\end{equation}	
 Therefore, if the initial data satisfies \eqref{assumption on ui in Lmi 1}, then there exists $\delta'\in(0,1)$ such that
		\begin{equation*}\label{}
				\|h_1\|^{\gamma_1}_1\|h_2\|^{\gamma_2}_1B
				\leq (1-\delta')x_0,
		\end{equation*}	
		i.e.,
		\begin{equation*}\label{}
				\theta\kappa^{m_1}_1\|u_1\|^{m_1}_{m_1}+(1-\theta)\kappa^{m_2}_2\|u_2\|^{m_2}_{m_2}\leq\frac{(1-\delta')x_0}
				{\kappa^{\gamma_1}_1\kappa^{\gamma_2}_2M^{\gamma_1}_1M^{\gamma_2}_2}.
		\end{equation*}	
	This is possible because the function $f(x)$ is increasing for all  $x\in(0,x_0)$.
	\end{proof}
	
In the case of $(\romannumeral2)$,  the equation $\alpha+\beta=2/d$ is valid for $m_1=m_2=m^*$, and it implies that  $(1-\alpha)/{m_1}+(1-\beta)/{m_2}=1$. Due to this fact, the results in Lemma \ref{lemma for uniform for u in Lm 1} are invalid. Indeed, we possess
	\begin{lemma}\label{lemma for uniform for u in Lm 2}	
		Given $\theta\in(0,1)$, $\alpha,\beta> 0$ satisfying $\alpha+\beta=2/d$, $\kappa_i>0$ in \eqref{definition of kappa}, $i=1,2$, and $T>0$. Let  $m_1=m_2=m^*$.
		Assume that initial data ${\textbf{u}^0}$ satisfies
		\begin{equation}\label{assumption on F}
				M^{\alpha}_1M^{\beta}_2<M^{-1}_c,
		\end{equation}
where $M_c$ is given by \eqref{definition of Mc}. Then there exists a constant $c>0$ such that the free energy solution ${\textbf{u}}$ of \eqref{system} with ${\textbf{u}^0}$ fulfills
		\begin{equation*}
				\|u_i(\cdot,t)\|_{m_i}\leq c\quad\text{for all}\,\,i=1,2,\,\,\text{and}\,\, t\in(0,T).
		\end{equation*}
	\end{lemma}
	\begin{proof}
We can begin by noting that the equation
		\begin{equation*}\label{}
				\mathcal{F}[u_1,u_2]=\frac{c_d}{\kappa_1\kappa_2}[\theta\|h_1\|^{m^*}_{m^*}+(1-\theta)\|h_2\|^{m^*}_{m^*}-\mathcal{H}[h_1,h_2]],\quad
		\end{equation*}	
which can be obtained through the transformation $h_i=\kappa_i u_i\in\Gamma_i$, $i=1,2$. From \eqref{inequality for H 1}, it follows that
		\begin{equation*}\label{}
				\mathcal{H}[h_1,h_2]		\leq\Lambda^*_{m^*,m^*,\theta}\|h_1\|^{\alpha}_1\|h_2\|^{\beta}_1
				[\theta\|h_1\|^{m^*}_{m^*}+(1-\theta)\|h_2\|^{m^*}_{m^*}],
		\end{equation*}	
		which implies that
		\begin{equation*}\label{}
			\begin{split}
				\frac{c_d}{\kappa_1\kappa_2}(1-\Lambda^*_{m^*,m^*,\theta}\|h_1\|^{\alpha}_1\|h_2\|^{\beta}_1)[\theta\|h_1\|^{m_1}_{m_1}+(1-\theta)\|h_2\|^{m_2}_{m_2}]\leq \mathcal{F}[u_1,u_2]\leq\mathcal{F}[u^0_1,u^0_2].
			\end{split}
		\end{equation*}
Once one possesses
				\begin{equation*}\label{condition for M1 and M2}
			\begin{split}
		M^{\alpha}_1M^{\beta}_2<(\kappa^{\alpha}_1\kappa^{\beta}_2\Lambda^*_{m^*,m^*,\theta})^{-1},					\end{split}
	\end{equation*}
the uniform bound of $u_i$ in $L^{m^*}(\mathbb{R}^d)$ will be achieved, $i=1,2$. It is important to observe that $(\kappa^{\alpha}_1\kappa^{\beta}_2\Lambda^*_{m^*,m^*,\theta})^{-1}=M_c$. In fact,
		\begin{equation*}\label{}
			\begin{split}
				\Lambda^*_{m^*,m^*,\theta}=&A\left(\theta,\frac{1-\alpha}{m^*}\right)\Lambda^*_{m^*,m^*}\\
				=&\left(\frac{1-\alpha}{m^*}\right)^{\frac{1-\alpha}{m^*}}\left(\frac{1-\beta}{m^*}\right)^{\frac{1-\beta}{m^*}}\theta^{-\frac{1-\alpha}{m^*}}(1-\theta)^{-\frac{1-\beta}{m^*}}\Lambda^*_{m^*,m^*}.
			\end{split}
		\end{equation*}	
Then by the definition of $\kappa_i>0$, $i=1,2$, it is clear that
		\begin{equation*}\label{}
			\begin{split}
				\kappa^{\alpha}_1\kappa^{\beta}_2\Lambda^*_{m^*,m^*,\theta}=&c^{\frac{\alpha+\beta}{2-m^*}}_d[(m^*-1)\theta]^{\frac{(m^*-1)\alpha+\beta}{(2-m^*)m^*}}[(m^*-1)(1-\theta)]^{\frac{\alpha+(m^*-1)\beta}{(2-m^*)m^*}}\\
				&\cdot\left(\frac{1-\alpha}{m^*}\right)^{\frac{1-\alpha}{m^*}}\left(\frac{1-\beta}{m^*}\right)^{\frac{1-\beta}{m^*}}\theta^{-\frac{1-\alpha}{m^*}}(1-\theta)^{-\frac{1-\beta}{m^*}}\Lambda^*_{m^*,m^*}\\
				=&c_d\frac{m^*-1}{m^*}\left(1-\alpha\right)^{\frac{1-\alpha}{m^*}}\left(1-\beta\right)^{\frac{1-\beta}{m^*}}\Lambda^*_{m^*,m^*}\\
				=&M_c.
			\end{split}
		\end{equation*}	
	\end{proof}

We can generalize the sufficient condition \eqref{assumption on F} stated in Lemma \ref{lemma for uniform for u in Lm 2}, guaranteeing that the solution is bounded, to more general case. 
		\begin{lemma}{\label{lemma on global existence in subcritical case2}}
		Let $M_1,M_2>0$, and let $\theta_0=M^{m^*}_1/(M^{m^*}_1+M^{m^*}_2)$. Assume that initial data ${\textbf{u}^0}$ fulfills  (\ref{total mass of solution}) and (\ref{assumption on initial data}). If $\boldsymbol{M}$ is subcritical in the sense that
		\begin{equation}\label{assumption on lambda2}
			\begin{split}
				\Sigma(\boldsymbol{M})&=c_d(m^*-1)\Pi^*_{\theta_0}M_1M_2/(M^{m^*}_1+M^{m^*}_2)<1,
			\end{split}
		\end{equation}
		then $\|u_i(\cdot,t)\|_{m_i}\leq c$ with some $c>0$, $i=1,2$, for all $t\in(0,T)$.
	\end{lemma}
	\begin{proof}	
The assumption \eqref{assumption on lambda2} implies the existence of a positive value $\eta>0$ such that
		\begin{equation*}\label{}
			\begin{split}
				\eta\in\left(0,\frac{1}{m^*-1}-c_d\Pi^*_{\theta_0}\frac{M_1M_2}{M^{m^*}_1+M^{m^*}_2}\right).
			\end{split}
		\end{equation*}	
Given that $(u_1/M_1,u_2/M_2)\in\Gamma_1\times\Gamma_2$, we can apply the definition of $\Pi^*_{\theta_0}$ to obtain the following inequality
		\begin{equation*}\label{}
			\begin{split}
				\mathcal{H}[u_1,u_2]\leq& \Pi^*_{\theta_0}\frac{M_1M_2}{M^{m^*}_1+M^{m^*}_2}(\|u_1\|^{m^*}_{m^*}+\|u_2\|^{m^*}_{m^*}).
			\end{split}
		\end{equation*}	
Inserting it into $\mathcal{F}$ yields
		\begin{equation*}\label{lower bound for F in criitcal case}
			\begin{split}
				\mathcal{F}[u_1,u_2]\geq& \left(\frac{1}{m^*-1}-c_d\Pi^*_{\theta_0}\frac{M_1M_2}{M^{m^*}_1+M^{m^*}_2}\right)(\|u_1\|^{m^*}_{m^*}+\|u_2\|^{m^*}_{m^*})\\
				\geq&\eta(\|u_1\|^{m^*}_{m^*}+\|u_2\|^{m^*}_{m^*}).
			\end{split}
		\end{equation*}	
		Therefore, the $L^{m^*}$-norm of solution is uniformly bounded. 
	\end{proof}
\begin{remark}\label{remark for pi and lambda}
We will now prove that \eqref{assumption on F} leads to \eqref{assumption on lambda2}. Applying Lemma \ref{lemma for Young inequality}, for every function $h_i\in L^1(\mathbb{R}^d)\cap L^{m^*}(\mathbb{R}^d)$, $i=1,2$, we observe that
		\begin{equation*}\label{*}
			\begin{split}
	\|h_1\|_1&\|h_2\|_1[\theta_0\|h_1\|^{-m^*}_1\|h_1\|^{m^*}_{m^*}+(1-\theta_0)\|h_2\|^{-m^*}_1\|h_2\|^{m^*}_{m^*}]\\
 \geq& \frac{1}{A(\theta_0,(1-\alpha)/m^*)}\|h_1\|^{\alpha}_1\|h_2\|^{\beta}_1\|h_1\|^{1-\alpha}_{m^*}\|h_2\|^{1-\beta}_{m^*}.
			\end{split}
		\end{equation*}	
This gives
		\begin{equation*}\label{*}
			\begin{split}
			&\frac{\mathcal{H}[h_1,h_2]}{\|h_1\|\|h_2\|_1[\theta_0\|h_1\|^{-m^*}_1\|h_1\|^{m^*}_{m^*}+(1-\theta_0)\|h_2\|^{-m^*}_1\|h_2\|^{m^*}_{m^*}]}\\
   &\leq A(\theta_0,(1-\alpha)/m^*)\frac{\mathcal{H}[h_1,h_2]}{\|h_1\|^{\alpha}_1\|h_2\|^{\beta}_1\|h_1\|^{1-\alpha}_{m^*}\|h_2\|^{1-\beta}_{m^*}}\\
			&\leq\frac{M^{m^*}_1+M^{m^*}_2}{m^*M^{1-\alpha}_1M^{1-\beta}_2}\left(1-\alpha\right)^{\frac{1-\alpha}{m^*}}\left(1-\beta\right)^{\frac{1-\beta}{m^*}}\Lambda^*_{m^*,m^*}\\
			&=\frac{M^{m^*}_1+M^{m^*}_2}{M_1M_2}\frac{M^\alpha_1M^{\beta}_2M_c}{c_d(m^*-1)}.
		\end{split}
	\end{equation*}	
In particular, 
			\begin{equation*}\label{}
		\Pi^*_{\theta_0}\leq\frac{M^{m^*}_1+M^{m^*}_2}{M_1M_2}\frac{M^\alpha_1M^{\beta}_2M_c}{c_d(m^*-1)}.
\end{equation*}	
Hence if \eqref{assumption on F} is valid, the inequality \eqref{assumption on lambda2} follows.
\end{remark}
	
	\noindent\textbf{{\emph{Proof of Theorem \ref{theorem on global existence} ($\romannumeral1$)}}}. Let $\theta$ be a real number such that $0 < \theta < 1$. Lemma \ref{lemma for uniform for u in Lm 1} provides a uniform bound of $\|u_i\|_{m_i}$, $i=1,2$, under the assumptions stated in Theorem \ref{theorem on global existence} ($\romannumeral1$). Hence  the global existence of free energy solution has been proved based on the fact that 
	$m_j/[(1+2/d)m_j-1]\in(1,m_i)$, $i\neq j\in\{1,2\}$, by Lemma \ref{lemma infity bound for u and w}.
	
	\noindent\textbf{{\emph{Proof of Theorem \ref{theorem on critical on the intersection point} ($\romannumeral1$)}}}.\,\,By applying Lemma \ref{lemma for uniform for u in Lm 2}, we get a uniform bound of $\|u_i\|_{m_i}$ if $\Sigma(\boldsymbol{M})<1$, $i=1,2,$ and hence the free energy functional exists globally.

	\section{Finite time Blow-up}\label{blow-up}
This section will focus on extablishing the blow-up solution. The first lemma provides some information about the second moment for free energy solution.
	\begin{lemma}{\label{lemma for second moment of functional}}
		Assume that initial data ${\textbf{u}^0}$ fulfills  (\ref{total mass of solution}) and (\ref{assumption on initial data}). Suppose that ${\textbf{u}}(t)$ is a free energy solution of (\ref{system}) with initial data ${\textbf{u}^0}$  on $t\in[0,T)$ with some $T\in(0,\infty]$. Then
		\begin{equation}{\label{second moment}}
\sum^2_{i=1}\int_{\mathbb{R}^d}|x|^2u_i(x)dx\leq\sum^2_{i=1}\int_{\mathbb{R}^d}|x|^2u^0_i(x)dx+\int^t_0\mathcal{I}(s)ds, \quad t\in(0,T),
		\end{equation}
		where the function $\mathcal{I}$ is given by
		\begin{equation}{\label{defnition of I}}
				\mathcal{I}(t)=2(d-2)\mathcal{F}[u_1,u_2](t)+2d\sum\limits^2_{i=1}\frac{m_i-2+2/d}{m_i-1}\int_{\mathbb{R}^d}u^{m_i}_i(x,t)dx
		\end{equation}
for $t\in(0,T)$.
	\end{lemma}
	\begin{proof}
		Utilizing cut-off function, \eqref{second moment} could be obtained for weak solutions in a formal way, see \cite[Lemma 6.2]{S2006-DIE} for example. Here, we adhere to a classical argument. Set
		$$\mathcal{S}[u_1,u_2]=\sum^2_{i=1}\int_{\mathbb{R}^d}|x|^2u_i(x)dx.$$
		A direct calculation reveals that
		\begin{equation*}{\label{}}
			\begin{split}
				\frac{d}{dt}\mathcal{S}[u_1,u_2]=&\sum\limits^2_{i=1}\int_{\mathbb{R}^d}|x|^2(\Delta u^{m_i}_i(x,t)-\nabla \cdot (u_i(x,t)\nabla v_j(x,t)))dx\\
				=&2d\sum^2_{i=1}\int_{\mathbb{R}^d}u^{m_i}_i(x,t)dx+2\sum\limits^2_{i=1}\int_{\mathbb{R}^d}u_i(x,t)\nabla v_j(x,t)\cdot xdx\\
				=&2d\sum^2_{i=1}\int_{\mathbb{R}^d}u^{m_i}_i(x,t)dx-2c_d(d-2)\iint_{\mathbb{R}^d\times\mathbb{R}^d}\frac{u_1(x,t)u_2(y,t)}{|x-y|^{d-2}}dxdy\\
				=&2d\sum^2_{i=1}\int_{\mathbb{R}^d}u^{m_i}_i(x,t)dx-2c_d(d-2)\mathcal{H}[u_1,u_2]\\
				=&2(d-2)\mathcal{F}[u_1,u_2]+2d\sum\limits^2_{i=1}\frac{m_i-2+2/d}{m_i-1}\|u_i\|^{m_i}_{m_i},\quad j\neq i,\quad t>0,
			\end{split}
		\end{equation*}
		which subsequently leads to the desired inequality after integrating with respect to time.	
	\end{proof}

The subsequent objective is to show the non-positivity of $\mathcal{I}$ for Cases $(\romannumeral1)$-$(\romannumeral2)$ in Section \ref{global existence}, which helps us to get the finite-time blow-up for some free energy solution.
	
	\begin{lemma}{\label{lemma for blow up in supercritical case2}}
Let $\theta\in(0,1)$. Assume that $m_1,m_2\in(1,m^*)$ satisfy $p\geq1$, $q\geq1$, $r>1$, and		 \begin{equation}\label{assumption on m1 and m2 case 1}
				\max\{m_1,m_2\}\leq2-\frac{2}{d}-\frac{d-2}{d}\left(1-\frac{1}{\frac{1-\alpha}{m_1}+\frac{1-\beta}{m_2}}\right).
		\end{equation}
		Let initial data ${\textbf{u}^0}$ fulfill  (\ref{total mass of solution}) and (\ref{assumption on initial data}). Suppose that \eqref{assumption on F 1} is valid and 
		\begin{equation}\label{assumption on u0 in the reverse}
				\theta\kappa^{m_1}_1\|u^0_1\|^{m_1}_{m_1}+(1-\theta)\kappa^{m_2}_2\|u^0_2\|^{m_2}_{m_2}>\frac{x_0}{	 \kappa^{\gamma_1}_1\kappa^{\gamma_2}_2M^{\gamma_1}_1M^{\gamma_2}_2}.
		\end{equation}
		Then the free energy solution with the initial data ${\textbf{u}^0}$ satisfies $\mathcal{I}(t)< 0$ for $t\in(0,T)$.
	\end{lemma}
	\begin{proof}
Let us begin by noting that the function $f(x)=x-\Lambda^*_{m_1,m_2,\theta}x^{\frac{1-\alpha}{m_1}+\frac{1-\beta}{m_2}}$ achieves its maximum at the point $x=x_0$. By using the transformation $h_i=\kappa_i u_i$ in  Lemma \ref{lemma for uniform for u in Lm 1}, where $\|h_i\|_1=\kappa_iM_i>0$, $i=1,2$, and considering that $\mathcal{F}$ decreases with time, we may infer from  \eqref{assumption on F 1} that
		\begin{equation*}\label{}
			\begin{split}
				2(d-2)\mathcal{F}[u_1,u_2]\leq& 2(d-2)\mathcal{F}[u^0_1,u^0_2]\\
				<&\frac{2c_d(d-2)}{\kappa^{1+\gamma_1}_1\kappa^{1+\gamma_2}_2M^{\gamma_1}_1M^{\gamma_2}_2} f(x_0)\\
				=&\frac{2c_d(d-2)}{\kappa^{1+\gamma_1}_1\kappa^{1+\gamma_2}_2M^{\gamma_1}_1M^{\gamma_2}_2}x_0 \left(1-\Lambda^*_{m_1,m_2,\theta}x^{\frac{1-\alpha}{m_1}+\frac{1-\beta}{m_2}-1}_0\right).
			\end{split}
		\end{equation*}	
Since $(1-\alpha)/m_1+(1-\beta)/m_2>1$, it follows from \eqref{assumption on F 1} and \eqref{inequality for f} that there exists small positive constant $\delta>0$ such that
		\begin{equation*}\label{}
			\begin{split}
				(1+\delta)f(x_0)
				\geq&f(\|h_1\|^{\gamma_1}_1\|h_2\|^{\gamma_2}_1B),
			\end{split}
		\end{equation*}
where $B=\theta\|h_1\|^{m_1}_{m_1}+(1-\theta)\|h_2\|^{m_2}_{m_2}$. Combining \eqref{assumption on u0 in the reverse} with  the non-increasing property of $f(x)$ for $x>x_0$,  it follows that there exists a positive constant $\delta'>0$ such that
		\begin{equation}\label{inequality for theta h1 in Lm1 and 1-theta h2 in Lm2}
			\begin{split}
				\|h_1\|^{\gamma_1}_1\|h_2\|^{\gamma_2}_1\left(\theta\|h_1\|^{m_1}_{m_1}+(1-\theta)\|h_2\|^{m_2}_{m_2}\right)>(1+\delta')x_0.
			\end{split}
		\end{equation}
For the sake of simplicity, let us assume that $m_2\leq m_1<2-2/d$. By utilizing \eqref{inequality for theta h1 in Lm1 and 1-theta h2 in Lm2}, we have
		\begin{equation*}
			\begin{split}
				2d&\|h_1\|^{\gamma_1}_1\|h_2\|^{\gamma_2}_1\sum\limits^2_{i=1}\frac{m_i-2+2/d}{m_i-1}\|u_i\|^{m_i}_{m_i}\\
				\leq&2d(m_1-2+2/d)\|h_1\|^{\gamma_1}_1\|h_2\|^{\gamma_2}_1\left[\frac{1}{(m_1-1)\kappa^{m_1}_1}\|h_1\|^{m_1}_{m_1}+\frac{1}{(m_2-1)\kappa^{m_2}_2}\|h_2\|^{m_2}_{m_2}\right]\\
				=&\frac{2c_dd(m_1-2+2/d)}{\kappa_1\kappa_2}\|h_1\|^{\gamma_1}_1\|h_2\|^{\gamma_2}_1\left[\theta\|h_1\|^{m_1}_{m_1}+(1-\theta)\|h_2\|^{m_2}_{m_2}\right]\\
				<&\frac{2c_dd(m_1-2+2/d)}{\kappa_1\kappa_2}(1+\delta')x_0.
			\end{split}
		\end{equation*}
Then 			
\begin{equation*}
\begin{split}
				\|h_1\|^{\gamma_1}_1&\|h_2\|^{\gamma_2}_1\left(2(d-2)\mathcal{F}[u_1,u_2]+2d\sum\limits^2_{i=1}\frac{m_i-2+2/d}{m_i-1}\|u_i\|^{m_i}_{m_i}\right)\\
				\leq&\frac{2c_d(d-2)}{\kappa_1\kappa_2}x_0 \left(1-\Lambda^*_{m_1,m_2,\theta}x^{\frac{1-\alpha}{m_1}+\frac{1-\beta}{m_2}-1}_0\right)+\frac{2c_dd(m_1-2+2/d)}{\kappa_1\kappa_2}(1+\delta')x_0\\
				=& \frac{2c_d(d-2)}{\kappa_1\kappa_2}x_0 \left(\frac{1}{\frac{1-\alpha}{m_1}+\frac{1-\beta}{m_2}}-\Lambda^*_{m_1,m_2,\theta}x^{\frac{1-\alpha}{m_1}+\frac{1-\beta}{m_2}-1}_0\right)\\
				&+\frac{2c_d(d-2)}{\kappa_1\kappa_2}x_0 \left(1-\frac{1}{\frac{1-\alpha}{m_1}+\frac{1-\beta}{m_2}}\right)+\frac{2c_dd(m_1-2+2/d)}{\kappa_1\kappa_2}(1+\delta')x_0\\
				=&\frac{2c_d(d-2)}{\kappa_1\kappa_2}x_0 \left(1-\frac{1}{\frac{1-\alpha}{m_1}+\frac{1-\beta}{m_2}}\right)+\frac{2c_dd(m_1-2+2/d)}{\kappa_1\kappa_2}(1+\delta')x_0\\
				<&0,
			\end{split}
\end{equation*}	
where we have used \eqref{assumption on m1 and m2 case 1}. Finally we achieve $\mathcal{I}(t)<0$ for $t\in(0,T)$.
	\end{proof}
	
	\begin{lemma}{\label{lemma for blow up in supercritical case3}}
Let $m_1=m_2=m^*$. Given $M_1,M_2>0$ satisfying
		\begin{equation}\label{assumption on u0 in the reverse 2}
	\Sigma(\boldsymbol{M})=	c_d(m^*-1)M_1M_2\Pi^*_{\theta_0}/(M^{m^*}_1+M^{m^*}_2)>1,
\end{equation}
where $\theta_0=M^{m^*}_1/(M^{m^*}_1+M^{m^*}_2)\in(0,1)$. Then one can construct  initial data  ${\textbf{u}^0}$ satisfying \eqref{total mass of solution}, \eqref{assumption on initial data}  and \eqref{assumption on u0 in the reverse 2} such that the corresponding free energy solution satisfies $\mathcal{I}(t)=2(d-2)\mathcal{F}[u_1,u_2](t)<0$ for $t\in(0,T)$.
	\end{lemma}
	\begin{proof}
Let us first pick $\eta\in(0,1)$ such that 
\begin{equation}\label{choice of eta}
1-\eta\Sigma(\boldsymbol{M})<0.
\end{equation}
Then there exists non-negative function $h_i\in L^{1}(\mathbb{R}^d)\cap L^{m^*}(\mathbb{R}^d)$, $i=1,2$, such that 	
\begin{equation*}
			\eta \Pi^{*}_{\theta_0}\leq\frac{\mathcal{H}[h_1,h_2]}{\|h_1\|_1\|h_2\|_1[\theta_0\| h_1\|^{-m^*}_1\|h_1\|^{m^*}_{m^*}+(1-\theta_0)\|h_2\|^{-m^*}_1\|h_2\|^{m^*}_{m^*}]}\leq \Pi^{*}_{\theta_0}.
	\end{equation*}
For $\mu>0$, let $\lambda_i=M_i\|h_i\|^{-1}_1\mu^d$, $i=1,2$. Denote $u^{0}_i(x)=\lambda_ih_i(\mu x)$. Then $\|u^0_i\|_1=M_i$, $\|u^0_i\|_{m^*}=\lambda_i\mu^{-\frac{d}{m^*}}\|h_i\|_{m^*}$, $i=1,2$, and $\mathcal{H}[u^0_1,u^0_2]=\lambda_1\lambda_2\mu^{-d-2}\mathcal{H}[h_1,h_2]$. Fix any $\mu>0$, then 
	\begin{equation*}
		\begin{split}
			\mathcal{F}[u^0_1,u^0_2]
			=&\frac{1}{m^*-1}\|u^0_1\|^{m^*}_{m^*}+\frac{1}{m^*-1}\|u^0_2\|^{m^*}_{m^*}-c_d\mathcal{H}[u^0_1,u^0_2]\\
   \leq&\frac{1}{m^*-1}\lambda^{m^*}_1\mu^{-d}\|h_1\|^{m^*}_{m^*}+\frac{1}{m^*-1}\lambda^{m^*}_2\mu^{-d}\|h_2\|^{m^*}_{m^*}-c_d\lambda_1\lambda_2\mu^{-d-2}\mathcal{H}[h_1,h_2]\\
   =&\frac{1}{m^*-1}M^{m^*}_1\|h_1\|^{-m^*}_1\mu^{d-2}\|h_1\|^{m^*}_{m^*}+\frac{1}{m^*-1}M^{m^*}_2\|h_2\|^{-m^*}_1\mu^{d-2}\|h_2\|^{m^*}_{m^*}\\
   &-c_d\eta M_1M_2\Pi^*_{\theta_0}\mu^{d-2}[\theta_0\|h_1\|^{-m^*}_1\|h_1\|^{m^*}_{m^*}+(1-\theta_0)\|h_2\|^{-m^*}_1\|h_2\|^{m^*}_{m^*}]\\
   =&\frac{\mu^{d-2}}{m^*-1}\left[M^{m^*}_1\|h_1\|^{-m^*}_1\|h_1\|^{m^*}_{m^*}
    +M^{m^*}_2\|h_2\|^{-m^*}_1\|h_2\|^{m^*}_{m^*}\right](1-\eta\Sigma(\boldsymbol{M})).
		\end{split}
	\end{equation*}
Due to the definition of $\theta_0$ and the choice of $\eta$ in \eqref{choice of eta}, then $\mathcal{F}[u^0_1,u^0_2]<0$. Let $(u_1,u_2)$ be the free energy solution of \eqref{system} with initial data $(u^0_1,u^0_2)$. The pair of $(u_1,u_2)$ fulfills  \begin{center}$\mathcal{I}(t)=2(d-2)\mathcal{F}[u_1,u_2](t)\leq 2(d-2)\mathcal{F}[u^0_1,u^0_2]< 0$ for $t\in(0,T)$.\end{center}	
The proof has been finished.

		\end{proof}

	\noindent\textbf{{\emph{Proof of Theorems \ref{theorem on global existence} ($\romannumeral2$) and \ref{theorem on critical on the intersection point} ($\romannumeral2$)}}}.\,\,	We will  prove the claims by the method of contradiction. Based on Lemmas \ref{lemma for blow up in supercritical case2} and \ref{lemma for blow up in supercritical case3}, we can deduce that there exists a finite time $T_0>0$ such that the corresponding free energy solution satisfies
	\begin{equation*}\label{finite-time for second moment}
		\begin{split}
			\lim\limits_{t\rightarrow T_0}\mathcal{S}[u_1,u_2]=0,
		\end{split}
	\end{equation*}	
as long as the assumptions in Theorems \ref{theorem on global existence} ($\romannumeral2$) or \ref{theorem on critical on the intersection point} ($\romannumeral2$) hold true. Subsequently, according to Lemma \ref{lemma for second moment of functional}, the second moment will eventually become negative after some time, which contradicts the non-negativity of $u_i,i=1,2$. Hence blow-up occurs.

\section*{Acknowledgments} JAC was supported by the Advanced Grant Nonlocal-CPD (Nonlocal PDEs for Complex Particle Dynamics: Phase Transitions, Patterns and Synchronization) of the European Research Council Executive Agency (ERC) under the European Union’s Horizon 2020 research and innovation programme (grant agreement No. 883363).
JAC was also partially supported by the Engineering and Physical Sciences Research Council (EPSRC) under the grant EP/V051121/1. JAC was also partially supported by the “Maria de Maeztu” Excellence Unit IMAG, reference CEX2020-001105-M, funded by the Spanish ministry of Science and Innovation MCIN/AEI/10.13039/501100011033/.
K. Lin is supported by the Natural Science Foundation of Sichuan Province (No.2022NSFSC1837).

\end{document}